\documentclass[11pt,oneside]{amsart}
\usepackage{eurosym,graphicx,amscd,color,amsmath,amsfonts,amssymb,amsthm,centernot,geometry,hyperref,mathtools}
\usepackage[initials]{amsrefs}
\usepackage[all]{xy}
\usepackage[norefs,nocites]{refcheck}
\usepackage{amsmath}
\usepackage{amsfonts}
\usepackage{amssymb}
\usepackage{graphicx}
\providecommand{\U}[1]{\protect\rule{.1in}{.1in}}

\newtheorem{theorem}{Theorem}

\newtheorem{corollary}{Corollary}

\newtheorem{lemma}{Lemma}
\newtheorem{proposition}{Proposition}
\newtheorem{remark}{Remark}

\begin{document}
\title[Complements of Unions: Insights on Spaceability and Applications]{Complements of Unions: Insights on Spaceability and Applications}
\author[Ara\'{u}jo]{G. Ara\'{u}jo}
\address[G. Ara\'{u}jo]{Departamento de Matem\'{a}tica \\
Universidade Estadual da Para\'{\i}ba \\
58.429-500 Campina Grande, Brazil.}
\email{gustavoaraujo@servidor.uepb.edu.br}
\author[Barbosa]{A. Barbosa}
\address[A. Barbosa]{Departamento de Matem\'atica \\
Universidade Federal da Para\'iba \\
58.051-900 Jo\~ao Pessoa, Brazil.}
\email{afsb@academico.ufpb.br}
\author[Raposo Jr.]{A. Raposo Jr.}
\address[A. Raposo Jr.]{Departamento de Matem\'atica \\
Universidade Federal do Maranh\~ao \\
65.085-580 S\~ao Lu\'is, Brazil.}
\email{anselmo.junior@ufma.br}
\author[G. Ribeiro]{G. Ribeiro$^{\ast}$}
\address[G. Ribeiro]{Departamento de Matem\'{a}tica \\
Universidade Federal da Para\'iba \\
58.051-900 Jo\~ao Pessoa, Brazil.}
\email{geivison.ribeiro@academico.ufpb.br}
\thanks{$^{*}$Corresponding author.}
\thanks{The first author acknowledges partial support from Grant 3024/2021, provided
by the Para\'{\i}ba State Research Foundation (FAPESQ). The third author
acknowledges support from Grant 406457/2023-9, funded by CNPq. The fourth
author acknowledges support from Grant 2022/1962, also provided by the
Para\'{\i}ba State Research Foundation (FAPESQ). This study received financial
support from the Coordena\c{c}\~{a}o de Aperfei\c{c}oamento de Pessoal de
N\'{\i}vel Superior - Brasil (CAPES) - Finance Code 001.}

\begin{abstract}
This paper presents two general criteria to determine spaceability results in
the complements of unions of subspaces. The first criterion applies to countable
unions of subspaces under specific conditions and is closely related to the
results of Kitson and Timoney in [J. Math. Anal. Appl. \textbf{378} (2011),
680-686]. This criterion extends and recovers some classical results in this
theory. The second criterion establishes sufficient conditions for the
complement of a union of Lebesgue spaces to be $\left(  \alpha,\beta\right)
$-spaceable, or not, even when they are not locally convex. We use this result
to characterize the measurable subsets having positive measure. Armed with these results, we have improved
existing results in environments such as: Lebesgue measurable function sets,
spaces of continuous functions, sequence spaces, nowhere H\"{o}lder function
sets, Sobolev spaces, non-absolutely summing operator spaces, and even sets of
functions of bounded variation.
\end{abstract}

\maketitle

\tableofcontents


\section{Preliminaries and background}

Since its emergence in 2005, the concepts of lineability and spaceability have
captured the interest of numerous researchers. Recently, the American
Mathematical Society incorporated these terms into its 2020 Mathematical
Subject Classification, citing them as 15A03 and 46B87. In essence, this idea
revolves around identifying significant algebraic frameworks within non-linear
subsets of a topological vector space, whenever it is possible. The purpose of this
article is to contribute to the elucidation of this contemporary trend by
presenting several general criteria that ensure the existence of these
aforementioned structures, with a particular focus on spaceability and
$(\alpha,\beta)$-spaceability. The definitions for these concepts are provided
below. For a comprehensive overview on lineability, we refer to the survey
\cite{aron}.

Some early instances of findings in this field can be attributed to V. I.
Gurariy (1935-2005), who established that the collection of continuous nowhere
differentiable functions on $\mathbb{R}$ contains (except for the null vector)
infinite-dimensional linear spaces. Moreover, Gurariy \textit{et al.} \cite{AGSS} demonstrated that the set of differentiable nowhere
monotone functions also contains (except for $0$) infinite-dimensional
linear spaces. Subsequently to these groundbreaking contributions, significant
progress has been made in establishing connections between various branches of
mathematics, including Linear Dynamics \cite{CSF2022}, Real and Complex
Analysis \cites{Bernal,AGSS,enflogurariyseoane2014}, Linear and Multilinear
Algebra, or even Operator Theory \cites{bertolotobotelhofavarojatoba2013},
Topology, Measure Theory \cite{Bernal}, Algebraic Geometry \cite{TAMS2020},
and Statistics \cite{stat}.

For instance, in \cite{Bernal,bc2} the authors obtained general results about
the spaceability of sets of the type $L_{p}(\Omega)\smallsetminus\bigcup
_{q\in\Gamma}L_{q}\left(  \Omega\right)  $, where $\Gamma$ is a certain subset
of $(0,\infty]$, $(\Omega,\mathcal{M},\mu)$ is an arbitrary measure space and
$L_{p}(\Omega)$ is the Lebesgue space of all measurable functions
$f:\Omega\rightarrow\mathbb{K}$ ($=\mathbb{R}$ or $\mathbb{C}$) such that
$|f|^{p}$ is $\mu$-integrable ($\mu$-essentially bounded if $p=\infty$) on
$\Omega$. In \cite{bbfp,BCP,BDFP,FPT,a2} the authors use different techniques
for checking spaceability into different \textcolor{black}{kinds} of sequence
sets. More precisely,
if $X$ is an infinite dimensional Banach space over $\mathbb{K}$ and
$X^{\mathbb{N}}$ denotes the space of all $X$-valued sequences, the authors of
\cite{bbfp,BCP,BDFP,FPT,a2} also studied, among others, the spaceability of the
following sets:

\begin{itemize}
\item $c_{0}(X)\smallsetminus\bigcup_{p\in[1,\infty)}\ell_{p}^{w}(X)$;
\item $\ell_{p}(X)\smallsetminus\bigcup_{q\in(1,p)}\ell_{q}^{w}(X)$;
\item $\ell_{p}^{u}(X)\smallsetminus\bigcup_{q\in(1,p)}\ell_{q}^{w}(X)$;
\item $\ell_{p}\left\langle X\right\rangle \smallsetminus\bigcup_{q\in
(1,p)}\ell_{q}^{w}(X)$;
\item $\ell_{p}^{w}(X)\smallsetminus\bigcup_{q\in(1,p)}\ell_{q}^{w}(X)$.
\end{itemize}

Above $\ell_{p}\left(  X\right)  $ is the space of all $p$-summable sequences
in $X^{\mathbb{N}}$, $\ell_{\infty}\left(  X\right)  $ is the space of all
bounded sequences in $X^{\mathbb{N}}$, $c(X)$ is the space of all convergent
sequences in $X^{\mathbb{N}}$, $c_{0}\left(  X\right)  $ is the space of all
sequences $\left(  x_{n}\right)  _{n=1}^{\infty}\in X^{\mathbb{N}}$ such that
$\lim x_{n}=0$, $c_{00}\left(  X\right)  $ is the space of all eventually null
sequences in $X^{\mathbb{N}}$, $\ell_{p}^{w}\left(  X\right)  $ is the space
of all weakly $p$-summable sequences in $X^{\mathbb{N}}$, $\ell_{p}^{u}\left(
X\right)  $ is the space of all unconditionally $p$-summable sequences in
$X^{\mathbb{N}}$, and $\ell_{p}\left\langle X\right\rangle $ is the space of all
Cohen $p$-summable sequences in $X^{\mathbb{N}}$.

We also highlight that very recently, Carmona Tapia \textit{et al.} \cite{a1} presented
results that represent a novelty in this area of research, which is the study
of lineability within the setting of Sobolev spaces. More precisely, they
proved that, for $1\leq p<\infty$, the set
\[
W^{m,p}(a,b)\smallsetminus\bigcup_{q\in(p,\infty)}W^{m,q}(a,b)
\]
is $\mathfrak{c}$-spaceable, where, for $1\leq p<\infty$ and $m\in\mathbb{N}$,
$W^{m,p}(a,b)$ represents the Sobolev space in the open interval $(a,b)$.

As a consequence of our main results (Theorems \ref{Theorem 2.4} and
\ref{t29}), we can extend and verify several classical results of this theory,
as well as we can ensure the existence, except for zero, of large algebraic
structures within, for example, spaces of Lebesgue measurable or continuous
functions, of several sequence spaces, of Holder or Sobolev spaces, spaces of
non-absolutely summing operators or even in the space of functions of bounded
variation. Among several other results, we can highlight that our theorems and
corollaries recover, extend or improve all the results already mentioned above.

Let us recall below some notions of lineability and spaceability that will be
useful for our purpose in this work. For any set $X$ we shall denote by
$\mathrm{card}(X)$ the cardinality of $X$; we also denote $\mathfrak{c} =
\mathrm{card}(\mathbb{R})$ and $\aleph_{0} = \mathrm{card} (\mathbb{N})$.
Assume that $E$ is a vector space and $\beta\leq\dim(E)$ is a cardinal number.
Then a subset $A \subset E$, is said to be:

\begin{itemize}
\item \textit{lineable} if there is an infinite dimensional vector subspace
$F$ of $E$ such that $F \smallsetminus\{0\} \subset A$, and

\item \textit{$\beta$-lineable} if there exists a vector subspace $F_{\beta}$
of $E$ with $\dim(F_{\beta}) = \beta$ and $F_{\beta}\smallsetminus\{0\}
\subset A$ (hence lineability means $\aleph_{0}$-lineability).
\end{itemize}

Also, if $\alpha$ is another cardinal number, with $\alpha<\beta$, then $A$ is
said to be:

\begin{itemize}
\item \textit{$(\alpha,\beta)$-lineable} if it is $\alpha$-lineable and for
every subspace $F_{\alpha}\subset E$ with $F_{\alpha}\subset A\cup\{0\}$ and
$\dim(F_{\alpha})=\alpha$, there is a subspace $F_{\beta}\subset E$ with
$\dim(F_{\beta})=\beta$ and
\begin{equation}
\label{ab}F_{\alpha}\subset F_{\beta}\subset A\cup\{0\}
\end{equation}
(hence $(0,\beta)$-lineability $\Leftrightarrow\ \beta$-lineability).
\end{itemize}

If $E$ is, in addition, a topological vector space, then $A$ is called

\begin{itemize}
\item \textit{$\beta$-spaceable} if $A\cup\{0\}$ contains a closed $\beta
$-dimensional linear subspace of $E$.

\item Moreover, if the subspace $F_{\beta}$ satisfying \eqref{ab} can always
be chosen closed, we say that $A$ is \textit{$(\alpha,\beta)$-spaceable} (see
\cite{FPT}).
\end{itemize}

The classical concept of lineability was coined by V. I. Gurariy in the early
2000's and it first appeared in print in \cite{AGSS,GQ}. The notion of
$(\alpha,\beta)$-lineability was introduced in \cite{FPT} (see
\cite{DR,DFPR,FPR}).

The paper is organized as follows. In Section \ref{sec2} we present our first
main result (Theorem \ref{Theorem 2.4}), a very generic criterion of
$(\alpha,\beta)$-spaceability in the context of complements of vector spaces,
and as a consequence, we present several corollaries that can be applied in
the most varied contexts. In Section \ref{sec3} we extend a result of
\cite{ABRR}, related to the Lebesgue spaces $L_{p}$ in $[0,1]$, to arbitrary
measure spaces. However, our extension is not valid for the non-locally convex
case ($0<p<1$), and this, in a way, tells us that Theorem \ref{Theorem 2.4}
needs a complementary result, which possibly can solve the cases not
contemplated by Theorem \ref{Theorem 2.4}. Therefore, in the same section, as
a complementary result of Theorem \ref{Theorem 2.4}, we prove Theorem
\ref{t29}, our other main result of this paper. In the subsections of Section
\ref{sec4}, we present several applications of the main results of the paper.
There are more than twenty applications, in several specific topological
vector spaces.

Just to give some concrete examples, we mention below three of the various
applications that we obtained as a consequence of our main results:

\begin{itemize}
\item[(i)] We prove a characterization of a measurable subset of positive
measure by means of $(\alpha,\beta)$-spaceability. More precisely, for
$m\in\mathbb{N}$, $p\in(0,\infty)$ and $X\subset\mathbb{R}^{m}$, we prove that
\[
\lambda(X)>0\text{ if and only if }L_{p}(X)\smallsetminus\textstyle\bigcup
_{q\in(p,\infty]}L_{q}(X)\text{ is }(\alpha,\mathfrak{c})\text{-spaceable for
all }\alpha<\aleph_{0},
\]
where $\lambda$ is the Lebesgue measure;

\item[(ii)] Let $\mathcal{C}_{0}\left[  0,\infty\right)  $ be the linear space
of the continuous functions $f: \left[  0,\infty\right)  \to\mathbb{R}$ such
that $\lim_{x\rightarrow\infty}f\left(  x\right)  =0$ endowed with the uniform
convergence norm (i.e., the supremum norm), and let $R\left[  0,\infty\right)  $ be the linear space of
the Riemann-integrable functions $f: \left[  0,\infty\right)  \to\mathbb{R}$
endowed with the norm $\sup_{x\geq0}\left\vert \int_{0}^{x}f\left(  t\right)
dt\right\vert $. In \cite{Bernal} the authors showed that the set
\begin{equation}
\label{bebebe}\left(  \mathcal{C}_{0}\left[  0,\infty\right)  \cap R\left[
0,\infty\right)  \right)  \smallsetminus{\bigcup_{p\in(0,\infty)}}L_{p}\left[
0,\infty\right)
\end{equation}
is spaceable in $\mathcal{C}_{0}\left[  0,\infty\right)  \cap R\left[
0,\infty\right)  $. As a consequence of our results, using a completely
different technique, we improve and complement the above result, showing that
the set in \eqref{bebebe} is $(\alpha,\mathfrak{c})$-spaceable for all
$\alpha<\aleph_{0}$.

\item[(iii)] Let $\mathcal{BVC}[0,1]$ and $\mathcal{AC}[0,1]$ be,
respectively, the set of all continuous functions $[0,1]\to\mathbb{R}$ of
bounded variation and the set of all absolutely continuous functions
$[0,1]\to\mathbb{R}$. As a consequence of our results, we can prove that the
set $\mathcal{BVC}[0,1]\smallsetminus\mathcal{AC}[0,1]$ is $(\alpha
,\mathfrak{c})$-spaceable in $\mathcal{BVC}[0,1]$ for all $\alpha<\aleph_{0}$.
If $\mathcal{C}[0,1]$ is the set of all continuous functions $[0,1]\to
\mathbb{R}$, then we can also prove that
$\mathcal{C}[0,1]\smallsetminus\mathcal{BVC}[0,1]$ is $(\alpha,\mathfrak{c}%
)$-spaceable in $\mathcal{C}[0,1]$ for all $\alpha<\aleph_{0}$.
\end{itemize}

\section{Spaceability in Complement of Unions}
\label{sec2}

In this section, we delve into a comprehensive analysis of spaceability in the
complements of unions of subspaces, aiming to answer a fundamental question:
\emph{what conditions are necessary and sufficient to ensure that the
complement of a union, not necessarily countable, is spaceable}?

So far, we do not have a complete set of conditions that characterize
spaceability in this complex context. However, works like that of Kitson and
Timoney in \cite{Kitson} offer intriguing insights, highlighting that under
certain conditions, the complement of a countable union of subspaces is
spaceable. Furthermore, we aim to understand more about the relationship
between the notions of \emph{spaceability} and $\left(  \alpha,\beta\right)
$-\emph{spaceability}. For example, the set of functions nowhere
differentiable in $C[0,1]$, denoted by $\mathcal{N}\mathcal{D}[0,1]$, is one
of those sets that can be described as the complement of a union of dense
subspaces. According to the literature (see \cite{aron, Fonf}), it is
spaceable, but it has not yet been determined whether it is specifically
$\left(  1,\mathfrak{c}\right)  $-spaceable or not.

At this point, we present the theorem below, which will be applied to countable
unions of subspaces under certain conditions and is closely related to results
verified in \cite{Kitson}.

\begin{theorem}
\label{Theorem 2.4}Let $V$ be a Banach space and $\left\{  W_{r}\right\}
_{r\in\mathbb{N}}$ be a family of subspaces of $V$. Assume that $%
{\textstyle\bigcup\nolimits_{r=1}^{\infty}}
W_{r}$ is a subspace of $V$ with infinite codimension. If for each
$r\in\mathbb{N}$, there exist a norm $\Vert\cdot\Vert_{W_{r}}$ for $W_{r}$ and
a constant $C_{r}>0$ such that $\left(  W_{r},\left\Vert \cdot\right\Vert
_{W_{r}}\right)  $ is Banach and $\left\Vert \cdot\right\Vert _{V}\leq
C_{r}\left\Vert \cdot\right\Vert _{W_{r}}$, then the complement
$V\smallsetminus%
{\textstyle\bigcup\nolimits_{r=1}^{\infty}}
W_{r}$ is
\color{black}%
$\left(  \alpha,\mathfrak{c}\right)  $-spaceable for all $\alpha<\aleph_{0}$.%
\color{black}%
\end{theorem}

\begin{proof}
Fix
\color{black}
$\alpha<\aleph_{0}$
\color{black}
and let $\mathcal{C}$ be a linearly independent subset of $V$ such that
\[
\mathrm{card}(\mathcal{C})=\alpha\text{
\color{black}%
and }\operatorname*{span}\left(  \mathcal{C}\right)  \cap\left(  \bigcup
_{r=1}^{\infty}W_{r}\right)  =\left\{  0\right\}  \text{.}%
\]
Consider the inclusion mapping $T_{1}\colon\left(  \operatorname*{span}\left(
\mathcal{C}\right)  ,\left\Vert \cdot\right\Vert _{V}\right)  \longrightarrow
\left(  V,\left\Vert \cdot\right\Vert _{V}\right)  $. Clearly, $T_{1}$ is
linear and continuous. For $r>1$, also consider the inclusion mapping
$T_{r}\colon\left(  W_{r},\left\Vert \cdot\right\Vert _{W_{r}}\right)
\longrightarrow\left(  V,\left\Vert \cdot\right\Vert _{V}\right)  $. It is
plain that $T_{r}$ is linear. Furthermore, $T_{r}$ is also continuous.
Indeed,
\[
\left\Vert T_{r}\left(  x\right)  \right\Vert _{V}=\left\Vert x\right\Vert
_{V}\leq C_{r}\left\Vert x\right\Vert _{W_{r}}\text{ for every }x\in\left(
W_{r},\left\Vert \cdot\right\Vert _{W_{r}}\right)  \text{.}%
\]
Now, let us turn our attention to the subspace $\operatorname*{span}\left[
\left(  \bigcup_{r=1}^{\infty}W_{r}\right)  \cup\operatorname*{span}\left(
\mathcal{C}\right)  \right]  $. If it is closed in $V$, then by \cite[Theorem
2.1]{Kitson} and the fact that $\operatorname*{span}\left[  \left(
\bigcup_{r=1}^{\infty}W_{r}\right)  \cup\operatorname*{span}\left(
\mathcal{C}\right)  \right]  =\left(
{\textstyle\bigcup\nolimits_{r=1}^{\infty}}
W_{r}\right)  +\operatorname*{span}\left(  \mathcal{C}\right)  $, we can
conclude that the complement $V\smallsetminus\left[  \left(
{\textstyle\bigcup\nolimits_{r=1}^{\infty}}
W_{r}\right)  +\operatorname*{span}\left(  \mathcal{C}\right)  \right]  $ is
spaceable. On the other hand, if $\operatorname*{span}\left[  \left(
\bigcup_{r=1}^{\infty}W_{r}\right)  \cup\operatorname*{span}\left(
\mathcal{C}\right)  \right]  $ is not closed in $V$, then by (see
\cite[Theorem 2.3]{Kitson}), the complement
\[
V\smallsetminus\left[  \left(  \bigcup_{r=1}^{\infty}W_{r}\right)
+\operatorname*{span}\left(  \mathcal{C}\right)  \right]
\]
is spaceable. In either case, there exists a closed infinite dimensional
subspace $F$ in $V$ such that
\[
F\cap\left[  \left(  \bigcup_{r=1}^{\infty}W_{r}\right)  +\operatorname*{span}%
\left(  \mathcal{C}\right)  \right]  =\left\{  0\right\}  \text{.}%
\]
Denoting $E:=F+\operatorname{span}(\mathcal{C})$, we now show that $(%
{\textstyle\bigcup\nolimits_{r=1}^{\infty}}
W_{r})\cap E=\left\{  0\right\}  $. Indeed, if $x\in(%
{\textstyle\bigcup\nolimits_{r=1}^{\infty}}
W_{r})\cap E$, then there exist $w\in F$ and $v\in\operatorname{span}%
(\mathcal{C})$ such that $x=w+v$. Hence, we have
\[
w=x-v\in\left[  (%
{\textstyle\bigcup\nolimits_{r=1}^{\infty}}
W_{r})+\operatorname{span}(\mathcal{C})\right]  \cap F=\{0\}\text{,}%
\]
which implies that $x=v\in(%
{\textstyle\bigcup\nolimits_{r=1}^{\infty}}
W_{r})\cap\operatorname{span}(\mathcal{C})=\{0\}$. Therefore, $(%
{\textstyle\bigcup\nolimits_{r=1}^{\infty}}
W_{r})\cap E=\{0\}$, that is,
\[
E\smallsetminus\{0\}\subset V\smallsetminus%
{\textstyle\bigcup\nolimits_{r=1}^{\infty}}
W_{r}\text{.}%
\]
Now, we claim that $E$ is closed in $V$. In fact, let $Q\colon
V\longrightarrow V/F$ be the quotient map of $V$ onto $V/F$. Since
$\operatorname{span}(\mathcal{C})$ is finite dimensional, the subspace
$Q\left(  \operatorname{span}(\mathcal{C})\right)  $ has finite dimension.
This implies that $Q\left(  \operatorname{span}(\mathcal{C})\right)  $ is
closed in $V/F$, and since $F+\operatorname{span}(\mathcal{C})=Q^{-1}\left(
Q\left(  \operatorname{span}(\mathcal{C})\right)  \right)  $ and $Q$ is
continuous, we conclude that $E=F+\operatorname{span}(\mathcal{C})$ is closed
in $V$. Therefore, due to the fact that $\dim\left(  E\right)  \geq
\mathfrak{c}$, the proof is complete.
\end{proof}

If $V$ is a Banach space and $W$ is a subspace of $V$, with $\dim(W)=\aleph_{0} $,
then the main result of \cite{AB} proves that $V\smallsetminus W$ is
$(\alpha,\dim(V))$-lineable for all $\alpha<\dim(V)$. However, from \cite{AB}
we cannot guarantee anything about the spaceability of this complement. As a
first application of Theorem \ref{Theorem 2.4}, we have the following result:

\begin{corollary}
\label{coro1}Let $V$ be a Banach space and let $W$ be a subspace of $V$ with
$\dim(W)=\aleph_{0}$. Then $V\smallsetminus W$ is $(\alpha,\mathfrak{c}
)$-spaceable for all $\alpha<\aleph_{0}$.
\end{corollary}

In the Corollary \ref{Corollary 2.7} bellow, the $(\alpha,\mathfrak{c}
)$-spaceability of $V\smallsetminus%
{\textstyle\bigcup\nolimits_{r=1}^{\infty}}
W_{r}$, $\alpha<\aleph_{0}$, it is a direct consequence of Theorem
\ref{Theorem 2.4}. If $\alpha\geq\aleph_{0}$, from \cite[Corollary 2.3]{fprs}
we can conclude that this complement is not $(\alpha,\mathfrak{c})$-spaceable.
So we have the following result:


\begin{corollary}
\label{Corollary 2.7} Let $V$ be a Banach space and $\left\{  W_{r}\right\}
_{r\in\mathbb{N}}$ be a family of subspaces of $V$. Assume that $%
{\textstyle\bigcup\nolimits_{r=1}^{\infty}}
W_{r}$ is an infinite codimensional subspace of $V$. If for each
$r\in\mathbb{N}$, there exists a norm $\Vert\cdot\Vert_{W_{r}}$ for $W_{r}$ and
a constant $C_{r}>0$ such that $\left(  W_{r},\left\Vert \cdot\right\Vert
_{W_{r}}\right)  $ is Banach and $\left\Vert \cdot\right\Vert _{V}\leq
C_{r}\left\Vert \cdot\right\Vert _{W_{r}}$, then
\[
V\smallsetminus\bigcup_{r=1}^{\infty}W_{r}\text{ is }\left(  \alpha
,\mathfrak{c}\right)  \text{-spaceable if and only if }\alpha<\aleph
_{0}\text{.}%
\]

\end{corollary}


\section{A first application and the need for a complementary result}

\label{sec3}

Let $(\Omega,\mathcal{M},\mu)$ be a measure space, with $\mu$ a positive
measure, and let $p\in(0,\infty]$. As usual, $L_{p}(\Omega)$ will denote the
vector space of all (Lebesgue classes of) measurable functions $f:\Omega
\rightarrow\mathbb{K}$ such that%
\[%
\begin{cases}
|f|^{p}\text{ is }\mu\text{-integrable on }\Omega\text{,} & \text{for
}0<p<\infty\text{,}\\
f\text{ is }\mu\text{-essentially bounded in }\Omega\text{,} & \text{for
}p=\infty\text{.}%
\end{cases}
\]
Recall that if $1\leq p<\infty$ (if $0<p<1$) then $L_{p}(\Omega)$ becomes a
Banach (quasi-Banach, resp.) space under the norm ($p$-norm, resp.)
\[
\Vert f\Vert_{p}=\left(  \int_{\Omega}|f|^{p}d\mu\right)  ^{\frac{1}{p}%
}\text{.}%
\]
If $p=\infty$, $L_{\infty}(\Omega)$ becomes a Banach space under the norm
\[
\Vert f\Vert=\inf\{M>0:|f|\leq M\ \mu\text{-almost everywhere in
\color{black}
}\Omega\text{%
\color{black}%
}\}\text{.}%
\]

F\'{a}varo \textit{et al.} \cite{fprs} proved that the set $L_{p}\left[  0,1\right]
\smallsetminus\bigcup_{q\in\left(  p,\infty\right)  }L_{q}\left[  0,1\right]
$, $p>0$, is not $(\alpha,\beta)$-spaceable for $\alpha\geq\aleph_{0}$,
regardless of the cardinal number $\beta$. In Corollary \ref{Corollary 2.7},
let $V=L_{p}[0,1]$ and $W_{r}:=L_{p+\frac{1}{r}}[0,1]$. Observe that $\left(
L_{p+\frac{1}{r}}[0,1],\left\Vert \cdot\right\Vert _{p+\frac{1}{r}}\right)  $
is Banach and that there is $C_{r}>0$ such that $\left\Vert \cdot\right\Vert
_{p}\leq C_{r}\left\Vert \cdot\right\Vert _{p+\frac{1}{r}}$ for every
$r\in\mathbb{N}$. Moreover note that
\[
\bigcup_{q\in(p,\infty)}L_{q}[0,1]=\bigcup_{r=1}^{\infty}L_{p+\frac{1}{r}%
}[0,1]\text{.}%
\]
Therefore, as a consequence of Corollary \ref{Corollary 2.7}, we can obtain that,
for $p\geq1$, the set
\begin{equation}
L_{p}[0,1]\smallsetminus\bigcup_{q\in(p,\infty)}L_{q}[0,1]\text{ is }\left(
\alpha,\mathfrak{c}\right)  \text{-spaceable if and only if }\alpha<\aleph
_{0}\text{.} \label{Lp-Lq}%
\end{equation}

It is already known (see \cite{ABRR}) that the above result is also valid for
$0<p<1$ (This discussion started in \cite{BFPS}). This is the first evidence
that, despite having many applications in the most varied contexts, Theorem
\ref{Theorem 2.4} needs a complementary result, which possibly can solve the
cases not contemplated by Theorem \ref{Theorem 2.4}. This will be done in
Theorem \ref{t29}, our other main result of this paper.

Despite not solving the case $0<p<1$, Theorem \ref{Theorem 2.4} can be used to
extend \eqref{Lp-Lq} to arbitrary measure spaces (see Corollary \ref{c5}). For
this, we will need the following results:

In what follows $(\Omega,\mathcal{M},\mu)$ denotes an arbitrary measure space.

\begin{proposition}
\label{p1}Let $(\Omega,\mathcal{M},\mu)$ be a measure space. If $f\in
L^{p}(\Omega)\cap L^{q}(\Omega)$, with $0<p\leq q\leq\infty$, then $f\in
L^{r}(\Omega)$ for each $p\leq r\leq q$ and
\[
\Vert f\Vert_{r}\leq\Vert f\Vert_{p}^{\theta}\cdot\Vert f\Vert_{q}^{1-\theta
}\text{,}
\]
where $0\leq\theta\leq1$ satisfies $\frac{1}{r}=\frac{\theta}{p}
+\frac{1-\theta}{q}$.
\end{proposition}


\begin{proposition}
\label{p2}Let $(\Omega,\mathcal{M},\mu)$ be a measure space. Given
$p\in(0,\infty]$ and $\emptyset\neq\Gamma\subset(0,\infty]$, with
$\Gamma\subset\lbrack p,\infty]$ or $\Gamma\subset(0,p]$, the subset $\left(
\bigcup_{q\in\Gamma}L_{q}(\Omega)\right)  \cap L_{p}(\Omega)$ is a subspace of
$L_{p}(\Omega)$.
\end{proposition}

\begin{proof}
Obviously,
\[
\left(  \bigcup_{q\in\Gamma}L_{q}(\Omega)\right)  \cap L_{p}(\Omega
)=\bigcup_{q\in\Gamma}\left(  L_{q}(\Omega)\cap L_{p}(\Omega)\right)
\text{.}
\]
Let $q_{1},q_{2}\in\Gamma$, with $q_{1}<q_{2}$. If $\Gamma\subset\left[
p,\infty\right]  $, it follows from Proposition \ref{p1} that $L_{q_{2}
}(\Omega)\cap L_{p}(\Omega)\subset L_{q_{1}}(\Omega)$, that is,
\[
L_{q_{2}}(\Omega)\cap L_{p}(\Omega)\subset L_{q_{1}}(\Omega)\cap L_{p}
(\Omega)\text{.}
\]
Analogously, if $\Gamma\subset(0,p]$, we conclude that
\[
L_{q_{1}}(\Omega)\cap L_{p}(\Omega)\subset L_{q_{2}}(\Omega)\cap L_{p}
(\Omega)\text{.}
\]
This shows that, in any case, $\{L_{q_{2}}(\Omega)\cap L_{p}(\Omega
)\}_{q\in\Gamma}$ is a well ordered family with respect to the inclusion
relation and, therefore, the union of its members is a subspace.
\end{proof}

\begin{remark}
\label{r2}If $p\geq1$ and
\[
\emptyset\neq\Gamma\subset%
\begin{cases}
(1,\infty]\text{,} & \text{if }p=1\text{,}\\[1pt]%
\left[  1,p\right)  \text{ or }\left(  p,\infty\right]  \text{,} & \text{if
}p>1\text{,}%
\end{cases}
\]
notice that
\[
L_{p}(\Omega)\smallsetminus\bigcup_{q\in\Gamma}L_{q}(\Omega)=L_{p}
(\Omega)\smallsetminus\bigcup_{q\in\Gamma}\left(  L_{q}(\Omega)\cap
L_{p}(\Omega)\right)  \text{,}
\]
and $L_{q}(\Omega)\cap L_{p}(\Omega)$ is a Banach space when endowed with the
norm $\Vert\cdot\Vert_{p,q}=\Vert\cdot\Vert_{p}+\Vert\cdot\Vert_{q}$. As we
saw in Proposition \ref{p2}, the collection $\{L_{q}(\Omega)\cap L_{p}
(\Omega)\}_{q\in\Gamma}$ is well ordered and this allows us to consider a
subset $\{q_{r}\}_{r\in\mathbb{N}}\subset\Gamma$ such that
\[
\bigcup_{q\in\Gamma}\left(  L_{q}(\Omega)\cap L_{p}(\Omega)\right)
=\bigcup_{r\in\mathbb{N}}\left(  L_{q_{r}}(\Omega)\cap L_{p}(\Omega)\right)
\text{.}
\]

\end{remark}

From Remark \ref{r2} we have the following result:

\begin{corollary}
\label{c5}Let $(\Omega,\mathcal{M},\mu)$ be a measure space. Let $p\geq1$ and
\[
\emptyset\neq\Gamma\subset%
\begin{cases}
(1,\infty]\text{,} & \text{if }p=1\text{,}\\[1pt]%
\left[  1,p\right)  \text{ or }\left(  p,\infty\right]  \text{,} & \text{if
}p>1\text{.}%
\end{cases}
\]
If $L_{p}(\Omega)\smallsetminus\bigcup_{q\in\Gamma}L_{q}(\Omega)$ is lineable,
then
\[
L_{p}(\Omega)\smallsetminus\bigcup_{q\in\Gamma}L_{q}(\Omega)
\]
is $(\alpha,\mathfrak{c})$-spaceable if and only if $\alpha<\aleph_{0}$.
\end{corollary}

\begin{proof}
From Remark \ref{r2} we know that:

\begin{enumerate}
\item[$\left(  i\right)  $] $(W_{q},\Vert\cdot\Vert_{p,q})=(L_{q}(\Omega)\cap
L_{p}(\Omega),\Vert\cdot\Vert_{p,q})$ is a Banach space for each $q\in\Gamma$;

\item[$\left(  ii\right)  $] $\Vert\cdot\Vert_{p}\leq\Vert\cdot\Vert_{p,q}$;

\item[$\left(  iii\right)  $] There is a subset $\{q_{r}\}_{r\in\mathbb{N}
}\subset\Gamma$ such that $\bigcup_{q\in\Gamma}W_{q}=\bigcup_{r\in\mathbb{N}
}W_{q_{r}}$.
\end{enumerate}

Hence, by Corollary \ref{Corollary 2.7} we can conclude that
\[
L_{p}(\Omega)\smallsetminus\bigcup_{q\in\Gamma}L_{q}(\Omega)=L_{p}
(\Omega)\smallsetminus\bigcup_{q\in\Gamma}W_{q}=L_{p}(\Omega)\smallsetminus
\bigcup_{r\in\mathbb{N}}W_{q_{r}}
\]
is $(\alpha,\mathfrak{c})$-spaceable for all $\alpha<\aleph_{0}$.
\end{proof}

Theorem \ref{t29} below plays the role of a first complementary result of
Theorem \ref{Theorem 2.4}, or, more precisely, a complement of Corollary
\ref{c5}. The proof of Theorem \ref{t29} is very similar to the proof of
the main result of \cite{ABRR}, where it is verified that \eqref{Lp-Lq} is
also valid for $0<p<1$. We recently realized that the main result of
\cite{ABRR} could be generalized and that one could extract from this
generalization several interesting applications. For the sake of completeness,
we will present the proof of Theorem \ref{t29}.

\begin{theorem}
\label{t29}Let $p\in(0,\infty)$ and $\emptyset\neq\Gamma\subset(0,\infty
]\smallsetminus\{p\}$. If $(\Omega,\mathcal{M},\mu)$ is a measure space such
that there exists $\{\Omega_{i}\}_{i\in\mathbb{N}}\subset\mathcal{M}$ with
$\Omega_{i}\cap\Omega_{j}=\emptyset$ to $i\neq j$, $\Omega=\bigcup
_{i\in\mathbb{N}}\Omega_{i}$ and $L_{p}(\Omega_{i})\smallsetminus
\operatorname{span}\left\{  \bigcup_{q\in\Gamma}L_{q}(\Omega_{i})\cap
L_{p}(\Omega_{i})\right\}  $ lineable for each $i\in\mathbb{N}$, then
$L_{p}(\Omega)\smallsetminus\bigcup_{q\in\Gamma}L_{q}(\Omega)$ is $(\alpha,\mathfrak{c})$-spaceable if and only if $\alpha<\aleph_{0}$.
\end{theorem}

\begin{remark}
It is worth mentioning that a set $\Gamma$ of indexes as in Theorem \ref{t29}
must, necessarily, be contained in $(0,p)$ or $(p,\infty]$, because if we have
$q,r\in\Gamma$, with $q<p<r$, then for $f\in L_{p}(\Omega_{i})$ we can
consider $g=f\cdot\mathcal{X}_{E}$ and $h=f\cdot\mathcal{X}_{\Omega
_{i}\smallsetminus E}$ with $E=\{x\in\Omega_{i}:|f(x)|>1\}$ \textrm{(}where
$\mathcal{X}_{A}$ denotes the characteristic function of $A$\textrm{)}.
Obviously $g\in L_{q}(\Omega_{i})\cap L_{p}(\Omega_{i})$ and $h\in
L_{r}(\Omega_{i})\cap L_{p}(\Omega_{i})$ with $f=g+h$. This shows us that
\begin{align*}
L_{p}(\Omega_{i})  &  =\left(  \bigcup_{q\in\Gamma\cap(0,p)}L_{q}(\Omega
_{i})\cap L_{p}(\Omega_{i})\right)  +\left(  \bigcup_{r\in\Gamma\cap
(p,\infty]}L_{r}(\Omega_{i})\cap L_{p}(\Omega_{i})\right) \\
&  =\operatorname{span}\left\{  \bigcup_{q\in\Gamma}L_{q}(\Omega_{i})\cap
L_{p}(\Omega_{i})\right\}  \text{.}%
\end{align*}

\end{remark}

\begin{proof}
[Proof of Theorem \ref{t29}]Let $g_{1},\ldots,g_{n}\in L_{p}(\Omega)$ be
linearly independent normalized vectors so that $\operatorname{span}%
\{g_{1},\ldots,g_{n}\}\smallsetminus\{0\}\subset L_{p}(\Omega)\smallsetminus
\bigcup_{q\in\Gamma}L_{q}(\Omega)$. For every $k\in\mathbb{N}$ consider a subspace 
$V_{k}$ of $L_{p}(\Omega_{k})$ so that
\[
L_{p}(\Omega_{k})=V_{k}\oplus\operatorname{span}\left\{  \bigcup_{q\in\Gamma
}L_{q}(\Omega_{k})\cap L_{p}(\Omega_{k})\right\}  \text{,}%
\]
and $P_{k}\colon L_{p}(\Omega_{k})\rightarrow V_{k}$ the projection of
$L_{p}(\Omega_{k})$ over $V_{k}$. Since \( P_{k}\left(\operatorname{span}\{g_{1}|_{\Omega{k}}, \ldots, g_{n}|_{\Omega{k}}\}\right) \) is a subspace of \( L_{p}(\Omega_{k}) \) of dimension at most \( n \), by the hypothesis that \( L_{p}(\Omega_{i}) \smallsetminus \operatorname{span}\left\{ \bigcup_{q\in\Gamma} L_{q}(\Omega_{i}) \cap L_{p}(\Omega_{i})\right\} \) is lineable, i.e.,
\[
\mathrm{codim}\left(\operatorname*{span}\{g_{1}|_{\Omega{k}}, \ldots, g_{n}|_{\Omega{k}}\}\right) = \mathrm{dim}\left(\mathrm{V}_k\right) \geq \aleph_{0},
\]
we can choose  a unit vector $\overset{\sim}{f_{k}}$ in
\[
V_{k}\smallsetminus P_{k}\left(  \operatorname{span}\{g_{1}|_{\Omega_{k}%
},\ldots,g_{n}|_{\Omega_{k}}\}\right).
\]
Let us prove that, for all $a_{1},\ldots,a_{n}\in\mathbb{K}$, we have
\begin{equation}
\tilde{f}_{k}+\sum_{i=1}^{n}a_{i}g_{i}|_{\Omega_{k}}\notin\bigcup_{q\in\Gamma
}L_{q}(\Omega_{k})\text{.} \label{eq}%
\end{equation}%
\color{black}%
Indeed, if there exist
\color{black}
$a_{1},\ldots,a_{n}\in\mathbb{K}$ such that $\tilde{f}_{k}+\sum_{i=1}^{n}%
a_{i}g_{i}|_{\Omega_{k}}\in\bigcup_{q\in\Gamma}L_{q}(\Omega_{k})$, since
\[
\tilde{f}_{k}+\sum_{i=1}^{n}a_{i}g_{i}|_{\Omega_{k}}=\tilde{f}_{k}%
+P_{k}\left(  \sum_{i=1}^{n}a_{i}g_{i}|_{\Omega_{k}}\right)  +\left(
-P_{k}\left(  \sum_{i=1}^{n}a_{i}g_{i}|_{\Omega_{k}}\right)  +\sum_{i=1}%
^{n}a_{i}g_{i}|_{\Omega_{k}}\right)  \text{,}%
\]
we would conclude that $\tilde{f}_{k}+P_{k}\left(  \sum_{i=1}^{n}a_{i}%
g_{i}|_{\Omega_{k}}\right)  =0$ and, hence, $\tilde{f}_{k}\in P_{k}\left(
\operatorname*{span}\{g_{1}|_{\Omega_{k}},\ldots,g_{n}|_{\Omega_{k}}\}\right)
$, which we know
\color{black}
does not
\color{black}
happen.

Define $\tilde{p}=1$ if $p\geq1$ and $\tilde{p}=p$ if $0<p<1$. Furthermore,
consider $f_{k}\in L_{p}(\Omega)\smallsetminus\bigcup_{q\in\Gamma}L_{q}%
(\Omega)$ where%
\[
\text{%
\color{black}%
}f_{k}:=\left\{
\begin{array}
[c]{ccc}%
0\text{,} & \text{in} & \Omega\smallsetminus\Omega_{k}\text{,}\\
\tilde{f}_{k}, & \text{in} & \Omega_{k}\text{.}%
\end{array}
\right.  \text{
\color{black}%
}%
\]
For $\left(  a_{i}\right)  _{i=1}^{\infty}\in\ell_{\tilde{p}}$,
\color{black}%
we have
\color{black}
\[
\Vert a_{1}g_{1}\Vert_{p}^{\tilde{p}}+\cdots+\Vert a_{n}g_{n}\Vert_{p}%
^{\tilde{p}}+\sum_{i=n+1}^{\infty}\Vert a_{i}f_{i-n}\Vert_{p}^{\tilde{p}}%
=\sum_{i=1}^{\infty}|a_{i}|^{\tilde{p}}<\infty\text{.}%
\]
Since $L_{p}(\Omega)$ is a Banach space for $p\geq1$ and a quasi-Banach space
for $0<p<1$, it follows that $a_{1}g_{1}+\cdots+a_{n}g_{n}+\sum_{i=n+1}%
^{\infty}a_{i}f_{i-n}\in L_{p}(\Omega)$ and, thus, the operator
\[
T\colon\ell_{\tilde{p}}\rightarrow L_{p}(\Omega)\text{, \ \ }\ \ T\left(
\left(  a_{i}\right)  _{i=1}^{\infty}\right)  =a_{1}g_{1}+\cdots+a_{n}%
g_{n}+\sum_{i=n+1}^{\infty}a_{i}f_{i-n}%
\]
is well defined. Note that its
linearity is evident, while its continuity follows from the closed graph theorem for F-spaces. For an arbitrary function $f\colon X\rightarrow\mathbb{K}$
whose domain is an arbitrary set $X$, let $\operatorname{supp}\left(
f\right)  =\{x\in X:f(x)\neq0\}$. Since $\operatorname{supp}(f_{i}%
)\cap\operatorname{supp}(f_{j})=\varnothing$ for $i\neq j$, we can conclude
that $T(\ell_{\tilde{p}})$ has infinite dimension.

Below we will show that there exist%
\color{black}%
s
\color{black}
a positive integer $m_{0}$ such that
\[
\left\{  g_{1}|_{%
{\textstyle\bigcup\nolimits_{i=1}^{m_{0}}}
\Omega_{i}},\ldots,g_{n}|_{%
{\textstyle\bigcup_{i=1}^{m_{0}}}
\Omega_{i}},f_{1}|_{%
{\textstyle\bigcup_{i=1}^{m_{0}}}
\Omega_{i}},\ldots,f_{m_{0}}|_{%
{\textstyle\bigcup_{i=1}^{m_{0}}}
\Omega_{i}}\right\}
\]
is a linearly independent set in $L_{p}(\bigcup_{i=1}^{m_{0}}\Omega_{i})$. We
first need to prove the following lemma:

\begin{lemma}
\label{lema}There exist%
\color{black}%
s
\color{black}
a positive integer $m_{1}$ such that
\[
\left\{  g_{1}|_{%
{\textstyle\bigcup_{i=1}^{m_{1}}}
\Omega_{i}},\ldots,g_{n}|_{%
{\textstyle\bigcup_{i=1}^{m_{1}}}
\Omega_{i}}\right\}
\]
is a linearly independent set in $L_{p}(\bigcup_{i=1}^{m_{1}}\Omega_{i})$.
\end{lemma}

\begin{proof}
[Proof of Lemma \ref{lema}]Fix $j\in\{1,\ldots,n\}$. Since $g_{j}%
|_{\textstyle\bigcup_{i=1}^{m}\Omega_{i}}\overset{m\rightarrow\infty}{\longrightarrow
}g_{j}$ in $L_{p}(\Omega)$, we have, $g_{j}|_{\textstyle\bigcup_{i=1}^{m}\Omega_{i}}%
\neq0$ for all large enough $m$. By contradiction, suppose that there is not a
positive integer $m_{1}$ such that $\left\{  g_{1}|_{\textstyle\bigcup_{i=1}^{m_{1}%
}\Omega_{i}},\ldots,g_{n}|_{\textstyle\bigcup_{i=1}^{m_{1}}\Omega_{i}}\right\}  $ is
linearly independent in $L_{p}(\bigcup_{i=1}^{m_{1}}\Omega_{i})$. Thus, the
set
\[
\left\{  g_{1}|_{%
{\textstyle\bigcup_{i=1}^{m}}
\Omega_{i}},\ldots,g_{n}|_{\textstyle\bigcup_{i=1}^{m}\Omega_{i}}\right\}
\]
is linearly dependent on $L_{p}(\bigcup_{i=1}^{m}\Omega_{i})$ for all
$m\in\mathbb{N}$. For each $m\in\mathbb{N}$, let
\[
\left\{  g_{1(m)}|_{\textstyle\bigcup_{i=1}^{m}\Omega_{i}},\ldots,g_{r(m)}%
|_{\textstyle\bigcup_{i=1}^{m}\Omega_{i}}\right\}
\]
be a smaller linearly dependent subset of $\left\{  g_{1}|_{\textstyle\bigcup_{i=1}%
^{m}\Omega_{i}},\ldots,g_{n}|_{\textstyle\bigcup_{i=1}^{m}\Omega_{i}}\right\}  $ and
define $\varphi\colon\mathbb{N}\rightarrow\mathcal{P}(\{1,\ldots,n\})$ by
$\varphi(m)=\{1(m),\ldots,r(m)\}$, where $\mathcal{P}(\{1,\ldots,n\})$ is the
set of all subsets of $\{1,\ldots,n\}$. Since $\operatorname*{card}%
(\mathcal{P}(\{1,\ldots,n\}))<\operatorname*{card}(\mathbb{N})=\aleph_{0}$,
there is $\{j_{1},\ldots,j_{r}\}\in\varphi(\mathbb{N})$ such that
$\operatorname*{card}(\varphi^{-1}(\{j_{1},\ldots,j_{r}\}))=\aleph_{0}$.
Define $\mathbb{N}^{\prime}:=\varphi^{-1}(\{j_{1},\ldots,j_{r}\})\subset
\mathbb{N}$ and note that
\[
\left\{  g_{1(m)}|_{\textstyle\bigcup_{i=1}^{m}\Omega_{i}},\ldots,g_{r(m)}%
|_{\textstyle\bigcup_{i=1}^{m}\Omega_{i}}\right\}  =\left\{  g_{j_{1}}|_{\textstyle\bigcup
_{i=1}^{m}\Omega_{i}},\ldots,g_{j_{r}}|_{\textstyle\bigcup_{i=1}^{m}\Omega_{i}}\right\}
\text{.}%
\]
Thus, if $m,\tilde{m}\in\mathbb{N}^{\prime}$ are such that $m<\tilde{m}$, then
there are $b_{1},\ldots,b_{r-1},\tilde{b}_{1},\ldots,\tilde{b}_{r-1}%
\in\mathbb{K}$ so that
\[
g_{j_{r}}|_{\textstyle\bigcup_{i=1}^{m}\Omega_{i}}=b_{1}g_{j_{1}}|_{%
{\textstyle\bigcup_{i=1}^{m}}
\Omega_{i}}+\cdots+b_{r-1}g_{j_{r-1}}|_{%
{\textstyle\bigcup_{i=1}^{m}}
\Omega_{i}}%
\]
and
\begin{equation}
g_{j_{r}}|_{%
{\textstyle\bigcup_{i=1}^{\tilde{m}}}
\Omega_{i}}=\tilde{b}_{1}g_{j_{1}}|_{%
{\textstyle\bigcup_{i=1}^{\tilde{m}}}
\Omega_{i}}+\cdots+\tilde{b}_{r-1}g_{j_{r-1}}|_{%
{\textstyle\bigcup_{i=1}^{\tilde{m}}}
\Omega_{i}}\text{.} \label{eq3}%
\end{equation}
Restricting \eqref{eq3} to $\bigcup_{i=1}^{m}\Omega_{i}$ we get
\begin{align*}
\tilde{b}_{1}g_{j_{1}}|_{%
{\textstyle\bigcup_{i=1}^{m}}
\Omega_{i}}+\cdots+\tilde{b}_{r-1}g_{j_{r-1}}|_{%
{\textstyle\bigcup_{i=1}^{m}}
\Omega_{i}}  &  =g_{j_{r}}|_{%
{\textstyle\bigcup_{i=1}^{m}}
\Omega_{i}}\\
&  =b_{1}g_{j_{1}}|_{%
{\textstyle\bigcup_{i=1}^{m}}
\Omega_{i}}+\cdots+b_{r-1}g_{j_{r-1}}|_{%
{\textstyle\bigcup_{i=1}^{m}}
\Omega_{i}}%
\end{align*}
and consequently,
\[
(\tilde{b}_{1}-b_{1})g_{j_{1}}|_{%
{\textstyle\bigcup_{i=1}^{m}}
\Omega_{i}}+\cdots+(\tilde{b}_{r-1}-b_{r-1})g_{j_{r-1}}|_{%
{\textstyle\bigcup_{i=1}^{m}}
\Omega_{i}}=0\text{.}%
\]
Since $\left\{  g_{j_{1}}|_{%
{\textstyle\bigcup_{i=1}^{m}}
\Omega_{i}},\ldots,g_{j_{r}}|_{%
{\textstyle\bigcup_{i=1}^{m}}
\Omega_{i}}\right\}  $ is a smaller linearly dependent subset of
\[
\left\{  g_{1}|_{%
{\textstyle\bigcup_{i=1}^{m}}
\Omega_{i}},\ldots,g_{n}|_{%
{\textstyle\bigcup_{i=1}^{m}}
\Omega_{i}}\right\}
\]
we can conclude that $\tilde{b}_{k}=b_{k}$, $k=1,\ldots,r-1$. Since
$m,\tilde{m}\in\mathbb{N}^{\prime}$ are arbitrary, we obtain
\[
g_{j_{r}}|_{%
{\textstyle\bigcup_{i=1}^{m}}
\Omega_{i}}=b_{1}g_{j_{1}}|_{%
{\textstyle\bigcup_{i=1}^{m}}
\Omega_{i}}+\cdots+b_{r-1}g_{j_{r-1}}|_{%
{\textstyle\bigcup_{i=1}^{m}}
\Omega_{i}}%
\]
for all $m\in\mathbb{N}^{\prime}$.

Therefore
\begin{align*}
g_{j_{r}}  &  =\lim\limits_{m\in\mathbb{N}^{\prime}}g_{j_{r}}|_{%
{\textstyle\bigcup_{i=1}^{m}}
\Omega_{i}}\\
&  =\lim\limits_{m\in\mathbb{N}^{\prime}}\left(  b_{1}g_{j_{1}}|_{%
{\textstyle\bigcup_{i=1}^{m}}
\Omega_{i}}+\cdots+b_{r-1}g_{j_{r-1}}|_{%
{\textstyle\bigcup_{i=1}^{m}}
\Omega_{i}}\right) \\
&  =b_{1}g_{j_{1}}+\cdots+b_{r-1}g_{j_{r-1}}\text{,}%
\end{align*}
which
\color{black}%
contradicts
\color{black}
the fact that $\{g_{1},\ldots,g_{n}\}$ is linearly independent.
\end{proof}

Let us return to the proof of Theorem \ref{t29}. Let us prove that the set
\[
\left\{  g_{1}|_{%
{\textstyle\bigcup_{i=1}^{m}}
\Omega_{i}},\ldots,g_{n}|_{%
{\textstyle\bigcup_{i=1}^{m}}
\Omega_{i}},f_{1}|_{%
{\textstyle\bigcup_{i=1}^{m}}
\Omega_{i}},\ldots,f_{m}|_{%
{\textstyle\bigcup_{i=1}^{m}}
\Omega_{i}}\right\}
\]
is linearly independent in $L_{p}(\bigcup_{i=1}^{m}\Omega_{i})$ for all $m\geq
m_{0}$, where
\[
m_{0}=\min\left\{  m:\left\{  g_{1}|_{{\textstyle\bigcup_{i=1}^{m}}\Omega_{i}}%
,\ldots,g_{n}|_{{\textstyle\bigcup_{i=1}^{m}}\Omega_{i}}\right\}  \ \text{is linearly
independent in}\ L_{p}\left(  \bigcup_{i=1}^{m}\Omega_{i}\right)  \right\}
\text{.}%
\]

Given $m\geq m_{0}$, let $b_{1},\ldots,b_{n},b_{n+1},\ldots,b_{n+m}%
\in\mathbb{K}$ such that
\[
b_{1}g_{1}|_{%
{\textstyle\bigcup_{i=1}^{m}}
\Omega_{i}}+\cdots+b_{n}g_{n}|_{%
{\textstyle\bigcup_{i=1}^{m}}
\Omega_{i}}+b_{n+1}f_{1}|_{%
{\textstyle\bigcup_{i=1}^{m}}
\Omega_{i}}+\cdots+b_{n+m}f_{m}|_{%
{\textstyle\bigcup_{i=1}^{m}}
\Omega_{i}}=0\text{,}%
\]
i.e.,
\begin{equation}
b_{1}g_{1}|_{%
{\textstyle\bigcup_{i=1}^{m}}
\Omega_{i}}+\cdots+b_{n}g_{n}|_{%
{\textstyle\bigcup_{i=1}^{m}}
\Omega_{i}}=-b_{n+1}f_{1}|_{%
{\textstyle\bigcup_{i=1}^{m}}
\Omega_{i}}-\cdots-b_{n+m}f_{m}|_{%
{\textstyle\bigcup_{i=1}^{m}}
\Omega_{i}}\text{.} \label{eq5}%
\end{equation}
Restricting the equality in \eqref{eq5} to $\Omega_{j}$, $j=1,\ldots,m$, we
have
\[
b_{1}g_{1}|_{\Omega_{j}}+\cdots+b_{n}g_{n}|_{\Omega_{j}}=-b_{n+j}\tilde{f}%
_{j},
\]
\color{black}%
i.e.,
\color{black}
$-b_{n+j}\tilde{f}_{j}=P_{j}\left(  b_{1}g_{1}|_{\Omega_{j}}+\cdots+b_{n}%
g_{n}|_{\Omega_{j}}\right)  \in P_{j}\left(  \operatorname*{span}%
\{g_{1}|_{\Omega_{j}},\ldots,g_{n}|_{\Omega_{j}}\}\right)  $, and we can
conclude that $b_{n+j}=0$.
\color{black}%
From
\color{black}
\eqref{eq5} we have
\[
b_{1}g_{1}|_{%
{\textstyle\bigcup_{i=1}^{m}}
\Omega_{i}}+\cdots+b_{n}g_{n}|_{%
{\textstyle\bigcup_{i=1}^{m}}
\Omega_{i}}=0\text{,}%
\]
and from the Lemma \ref{lema} we obtain $b_{1}=\cdots=b_{n}=0$.%

\color{black}%
Now, let us see that
\color{black}
\[
\overline{T(\ell_{\tilde{p}})}\smallsetminus\{0\}\subset L_{p}(\Omega
)\smallsetminus\bigcup_{q\in\Gamma}L_{q}(\Omega)\text{.}%
\]
Indeed,
\color{black}%
given
\color{black}
$h\in\overline{T(\ell_{\tilde{p}})}\smallsetminus\{0\}$,
\color{black}%
let
\color{black}%
$\left(  a_{i}^{(k)}\right)  _{i=1}^{\infty}\in\ell_{\tilde{p}}$
($k\in\mathbb{N}$) be such that
\[
T\left(  \left(  a_{i}^{(k)}\right)  _{i=1}^{\infty}\right)  \overset
{k\rightarrow\infty}{\longrightarrow}h\ \text{in }L_{p}(\Omega)\text{.}%
\]
It is not difficult to see that $T\left(  \left(  a_{i}^{(k)}\right)
_{i=1}^{\infty}\right)  |_{X}\overset{k\rightarrow\infty}{\longrightarrow
}h|_{X}$ in $L_{p}(X)$ for any measurable subset $X$ of $\Omega$. In order to
go further, the strategy
\color{black}%
will
\color{black}
be to prove that there is a sequence of scalars $\left(  a_{i}\right)
_{i=1}^{\infty}$ such that
\[
a_{1}g_{1}+\cdots+a_{n}g_{n}+\sum_{i=1}^{\infty}a_{n+i}f_{i}=h.
\]%
\color{black}%
In fact, for a fixed $m\geq m_{0}$, note that
\color{black}
\begin{align*}
a_{1}^{(k)}g_{1}|_{%
{\textstyle\bigcup_{i=1}^{m}}
\Omega_{i}}+\cdots+a_{n}^{(k)}g_{n}|_{%
{\textstyle\bigcup_{i=1}^{m}}
\Omega_{i}}  &  +a_{n+1}^{(k)}f_{1}|_{%
{\textstyle\bigcup_{i=1}^{m}}
\Omega_{i}}+\cdots+a_{n+m}^{(k)}f_{m}|_{%
{\textstyle\bigcup_{i=1}^{m}}
\Omega_{i}}\\
&  =T\left(  \left(  a_{i}^{(k)}\right)  _{i=1}^{\infty}\right)  |_{%
{\textstyle\bigcup_{i=1}^{m}}
\Omega_{i}}\overset{k\rightarrow\infty}{\longrightarrow}h|_{%
{\textstyle\bigcup_{i=1}^{m}}
\Omega_{i}}\text{.}%
\end{align*}
Moreover, note that $\operatorname{span}\left\{  g_{1}|_{%
{\textstyle\bigcup_{i=1}^{m}}
\Omega_{i}},\ldots,g_{n}|_{%
{\textstyle\bigcup_{i=1}^{m}}
\Omega_{i}},f_{1}|_{%
{\textstyle\bigcup_{i=1}^{m}}
\Omega_{i}},\ldots,f_{m}|_{%
{\textstyle\bigcup_{i=1}^{m}}
\Omega_{i}}\right\}  $ is finite dimensional in $L_{p}(\bigcup_{i=1}^{m}%
\Omega_{i})$. Since each finite dimensional subspace of a topological vector
space is closed, there are scalars $a_{1}(m),\ldots$, $a_{n+m}(m)$ such that
\begin{align*}
&  h|_{\textstyle\bigcup_{i=1}^{m}\Omega_{i}}\\
&  =a_{1}(m)g_{1}|_{\textstyle\bigcup_{i=1}^{m}\Omega_{i}}+\cdots
+a_{n}(m)g_{n}|_{\textstyle\bigcup_{i=1}^{m}\Omega_{i}}+a_{n+1}(m)f_{1}%
|_{\textstyle\bigcup_{i=1}^{m}\Omega_{i}}+\cdots+a_{n+m}(m)f_{m}%
|_{\textstyle\bigcup_{i=1}^{m}\Omega_{i}}.
\end{align*}
\color{black}%
The same reasoning can be applied to $\tilde{m}>m$, and therefore
\color{black}
\begin{align}
h|_{%
{\textstyle\bigcup_{i=1}^{m}}
\Omega_{i}}  &  =a_{1}(m)g_{1}|_{%
{\textstyle\bigcup_{i=1}^{m}}
\Omega_{i}}+\cdots+a_{n}(m)g_{n}|_{%
{\textstyle\bigcup_{i=1}^{m}}
\Omega_{i}}+a_{n+1}(m)f_{1}|_{%
{\textstyle\bigcup_{i=1}^{m}}
\Omega_{i}}\nonumber\\
&  \qquad\qquad+\cdots+a_{n+m}(m)f_{m}|_{%
{\textstyle\bigcup_{i=1}^{m}}
\Omega_{i}}\text{,} \label{beq5}%
\end{align}
and
\begin{align}
h|_{%
{\textstyle\bigcup_{i=1}^{\tilde{m}}}
\Omega_{i}}  &  =a_{1}(\tilde{m})g_{1}|_{%
{\textstyle\bigcup_{i=1}^{\tilde{m}}}
\Omega_{i}}+\cdots+a_{n}(\tilde{m})g_{n}|_{%
{\textstyle\bigcup_{i=1}^{\tilde{m}}}
\Omega_{i}}+a_{n+1}(\tilde{m})f_{1}|_{%
{\textstyle\bigcup_{i=1}^{\tilde{m}}}
\Omega_{i}}\nonumber\\
&  \qquad\qquad+\cdots+a_{n+\tilde{m}}(\tilde{m})f_{\tilde{m}}|_{%
{\textstyle\bigcup_{i=1}^{\tilde{m}}}
\Omega_{i}}\text{.} \label{beq6}%
\end{align}
Restricting (\ref{beq6}) to $\bigcup_{i=1}^{m}\Omega_{i}$ and comparing with
(\ref{beq5}) we get
\begin{align*}
&  a_{1}(\tilde{m})g_{1}|_{%
{\textstyle\bigcup_{i=1}^{m}}
\Omega_{i}}+\cdots+a_{n}(\tilde{m})g_{n}|_{%
{\textstyle\bigcup_{i=1}^{m}}
\Omega_{i}}+a_{n+1}(\tilde{m})f_{1}|_{%
{\textstyle\bigcup_{i=1}^{m}}
\Omega_{i}}+\cdots+a_{n+m}(\tilde{m})f_{m}|_{%
{\textstyle\bigcup_{i=1}^{m}}
\Omega_{i}}\\
&  =h|_{%
{\textstyle\bigcup_{i=1}^{m}}
\Omega_{i}}\\
&  =a_{1}(m)g_{1}|_{%
{\textstyle\bigcup_{i=1}^{m}}
\Omega_{i}}+\cdots+a_{n}(m)g_{n}|_{%
{\textstyle\bigcup_{i=1}^{m}}
\Omega_{i}}+a_{n+1}(m)f_{1}|_{%
{\textstyle\bigcup_{i=1}^{m}}
\Omega_{i}}+\cdots+a_{n+m}(m)f_{m}|_{%
{\textstyle\bigcup_{i=1}^{m}}
\Omega_{i}}\text{.}%
\end{align*}
Since the set $\left\{  g_{1}|_{%
{\textstyle\bigcup_{i=1}^{m}}
\Omega_{i}},\ldots,g_{n}|_{%
{\textstyle\bigcup_{i=1}^{m}}
\Omega_{i}},f_{1}|_{%
{\textstyle\bigcup_{i=1}^{m}}
\Omega_{i}},\ldots,f_{m}|_{%
{\textstyle\bigcup_{i=1}^{m}}
\Omega_{i}}\right\}  $ is linearly independent, we obtain
\[
a_{j}(m)=a_{j}(\tilde{m})\text{ for every }j=1,\ldots,n+m\text{.}%
\]
Thus we conclude that there is a sequence of scalars $(a_{i})_{i=1}^{\infty}$
\color{black}%
such that, for each $m\geq m_{0}$,
\color{black}
\begin{align*}
\left.  \left(  a_{1}g_{1}+\cdots+a_{n}g_{n}+\sum_{i=1}^{\infty}a_{n+i}%
f_{i}\right)  \right\vert _{%
{\textstyle\bigcup_{i=1}^{m}}
\Omega_{i}}  &  =(a_{1}g_{1}+\cdots+a_{n}g_{n})|_{%
{\textstyle\bigcup_{i=1}^{m}}
\Omega_{i}}+\left.  \left(  \sum_{i=1}^{m}a_{n+i}f_{i}\right)  \right\vert _{%
{\textstyle\bigcup_{i=1}^{m}}
\Omega_{i}}\\
&  =h|_{%
{\textstyle\bigcup_{i=1}^{m}}
\Omega_{i}}%
\end{align*}%
\color{black}%
and so we finally have
\color{black}
\[
a_{1}g_{1}+\cdots+a_{n}g_{n}+\sum_{i=1}^{\infty}a_{n+i}f_{i}=h\text{.}%
\]
Since $h\neq0$, it follows that $0\neq\left(  a_{i}\right)  _{i=1}^{\infty}$.
Therefore, if $a_{n+i}=0$ for all $i\in\mathbb{N}$, we have
\[
h=a_{1}g_{1}+\cdots+a_{n}g_{n}\in\operatorname*{span}\{g_{1},\ldots
,g_{n}\}\smallsetminus\{0\}\subset L_{p}(\Omega)\smallsetminus\bigcup
_{q\in\Gamma}L_{q}(\Omega)\text{.}%
\]
On the other hand, if $a_{n+i}\neq0$ for some $i\in\mathbb{N}$, from
\eqref{eq} we obtain
\[
\frac{1}{a_{n+i}}h|_{\Omega_{i}}=\tilde{f}_{i}+\frac{1}{a_{n+i}}\left(
a_{1}g_{1}+\cdots+a_{n}g_{n}\right)  |_{\Omega_{i}}\notin\bigcup_{q\in\Gamma
}L_{q}(\Omega_{i})\text{.}%
\]
Consequently, $h\notin\bigcup_{q\in\Gamma}L_{q}(\Omega)$ and the result is done.
\end{proof}

\section{Other applications}

\label{sec4}

\subsection{$L_{p}$ spaces and continuous functions}

The Corollaries \ref{co4}, \ref{co5}, \ref{co6} and \ref{co7} below are all
consequences of Theorem \ref{Theorem 2.4} (more precisely, Corollaries
\ref{co4} and \ref{co6} are consequences of Corollary \ref{coro1}, and
Corollaries \ref{co5} and \ref{co7} of Corollary \ref{Corollary 2.7}). Although they involve classical spaces that have been extensively studied within this theory, to the best of our knowledge, these results are new. From now on, for each $k=0,1,2,\ldots
$, $\mathcal{C}^{k}[0,1]$ will denote the linear space of the $\mathbb{K}
$-valued functions over $\left[  0,1\right]  $ whose derivative of order $k$
is continuous, endowed with the norm
\[
f\longmapsto\sum_{j=0}^{k}\left\Vert f^{\left(  j\right)  }\right\Vert
_{\infty}\text{,}
\]
where $\left\Vert \cdot\right\Vert _{\infty}$ denotes the uniform convergence
norm and, as usual, $f^{\left(  0\right)  }=f$ and, if $j\geq1$, $f^{\left(
j\right)  }$ is the $j$-th derivative of $f$. Also, as usual, $\mathcal{C}
^{0}\left[  0,1\right]  =\mathcal{C}\left[  0,1\right]  $ denotes the linear
space of continuous $\mathbb{K}$-valued functions over $\left[  0,1\right]  $
and $\mathcal{P}[0,1]$ will indicate the space of polynomials defined over
$[0,1]$ endowed with the uniform convergence norm. Furthermore, making
$I=\left[  0,\infty\right)  $, $\mathcal{C}_{0}\left(  I\right)  $ will denote
the linear space of the continuous functions $f\colon I\longrightarrow
\mathbb{R}$ such that
\[
\lim_{x\rightarrow\infty}f\left(  x\right)  =0
\]
endowed with the uniform convergence norm and $R\left(  I\right)  $ will
denote the linear space of the Riemann-integrable functions $f\colon
I\longrightarrow\mathbb{R}$ endowed with the norm
\[
f\longmapsto\sup_{x\geq0}\left\vert \int_{0}^{x}f\left(  t\right)
\,dt\right\vert \text{.}
\]

\begin{corollary}
\label{co4} The set $\mathcal{C}[0,1]\smallsetminus\mathcal{P}[0,1]$ is $(\alpha
,\mathfrak{c})$-spaceable in $\mathcal{C}[0,1]$ if and only if $\alpha<\aleph_{0}$.
\end{corollary}

\begin{corollary}
\label{co5}Given $k\in\mathbb{N}$, the set $\mathcal{C}[0,1]\smallsetminus
\mathcal{C}^{k}[0,1]$ is $(\alpha,\mathfrak{c})$-spaceable in $\mathcal{C}
[0,1]$ if and only if $\alpha<\aleph_{0}$.
\end{corollary}

\begin{corollary}
\label{co6}Given $p\in\lbrack1,\infty]$, $L_{p}[0,1]\smallsetminus
\mathcal{P}[0,1]$ is $(\alpha,\mathfrak{c})$-spaceable in $L_{p}[0,1]$ if and only if
$\alpha<\aleph_{0}$.
\end{corollary}

\begin{corollary}
\label{co7}Given $k\in\mathbb{N}\cup\{0\}$ and $p\in\lbrack1,\infty]$, the set
$L_{p}[0,1]\smallsetminus\mathcal{C}^{k}[0,1]$ is $(\alpha,\mathfrak{c}
)$-spaceable in $L_{p}[0,1]$ if and only if $\alpha<\aleph_{0}$.
\end{corollary}

Let
$(\Omega,\mathcal{M},\mu)$ be a measure space. In \cite{bc2} the authors obtained general results about the spaceability of sets of the type $L_{p}(\Omega)\smallsetminus\bigcup_{q\in\Gamma}L_{q}\left(\Omega\right)$, where $\Gamma$ is a certain subset of $[1,\infty]$ that does not contain $p$.

\begin{theorem}[Bernal-Gonz\'{a}lez, Ord\'{o}\~{n}ez Cabrera \cite{bc2}]\label{t31}
Let $(\Omega,\mathcal{M},\mu)$ be a measure space, and consider the conditions
\begin{itemize}
\item[($\alpha$)] $\inf\{\mu(A);\ A\in\mathcal{M},\ \mu(A)>0\}=0$;
\item[($\beta$)] $\sup\{\mu(A);\ A\in\mathcal{M},\ \mu(A)<\infty\}=\infty$.
\end{itemize}
Then the following assertions hold:
\begin{itemize}
\item[(1)] if $1\leq p<\infty$, then $L_{p}(\Omega)\smallsetminus\bigcup
_{q\in(p,\infty]}L_{q}(\Omega)$ is spaceable if and only if $\left(
\alpha\right)  $ holds;
\item[(2)] if $1<p\leq\infty$, then $L_{p}(\Omega)\smallsetminus\bigcup
_{q\in\lbrack1,p)}L_{q}(\Omega)$ is spaceable if and only if $\left(
\beta\right)  $ holds;
\item[(3)] if $1<p<\infty$, then $L_{p}(\Omega)\smallsetminus\bigcup
_{q\in\lbrack1,\infty]\smallsetminus\{p\}}L_{q}(\Omega)$ is spaceable if and
only if both $\left(  \alpha\right)  $ and $\left(  \beta\right)  $ holds.
\end{itemize}
\end{theorem}

As a consequence of Corollary \ref{c5}, we can say that the sets in $(1)$ and
$(2)$ of Theorem \ref{t31} are $(\alpha,\mathfrak{c})$-spaceable for every
$\alpha<\aleph_{0}$.

In \cite{Bernal} the authors mention that the arguments in \cite{bc2} can be
mimicked to provide the non locally convex version of Theorem \ref{t31}.

\begin{theorem}[Bernal-Gonz\'{a}lez, Ord\'{o}\~{n}ez Cabrera \cite{Bernal}]\label{t9} 
Let $(\Omega,\mathcal{M},\mu)$ be a measure space, and let conditions ($\alpha$) and ($\beta$) be the same as in Theorem \ref{t31}. If $p\in(0,\infty)$, then the following statements hold:
\begin{itemize}
\item[(1')] $L_{p}(\Omega)\smallsetminus%
{\textstyle\bigcup_{q\in(p,\infty]}}
L_{q}(\Omega)$ is spaceable if and only if $\left(  \alpha\right)  $ holds;
\item[(2')] $L_{p}(\Omega)\smallsetminus%
{\textstyle\bigcup_{q\in(0,p)}}
L_{q}(\Omega)$ is spaceable if and only if $\left(  \beta\right)  $ holds;
\item[(3')] $L_{p}(\Omega)\smallsetminus%
{\textstyle\bigcup_{q\in(0,\infty]\smallsetminus\{p\}}}
L_{q}(\Omega)$ is spaceable if and only if both $\left(  \alpha\right)  $ and
$\left(  \beta\right)  $ holds.
\end{itemize}
\end{theorem}

Notice that if $\Omega$ is some of the sets $(0,\infty)$, $\mathbb{R}^{m}$,
$m\in\mathbb{N}$, $\mathcal{L}$ is the Lebesgue $\sigma$-algebra and $\lambda$
is the Lebesgue measure, then the measure space $\left(  \Omega,\mathcal{L}
,\lambda\right)  $ satisfies the conditions $(\alpha)$ and $(\beta)$ of
Theorems \ref{t31} and \ref{t9}.

\begin{corollary}
Given $p\in(0,\infty)$ and $m\in\mathbb{N}$, the set
\[
L_{p}(\mathbb{R}^{m})\smallsetminus\bigcup_{q\in(0,p)}L_{q}(\mathbb{R}^{m})
\]
is $(\alpha,\mathfrak{c})$-spaceable in $L_{p}(\mathbb{R}^{m})$ if and only if
$\alpha<\aleph_{0}$.
\end{corollary}

\begin{proof}
It is sufficient to consider a decomposition $\mathbb{R}^{m}=%
{\textstyle\bigcup_{i\in\mathbb{N}}}
X_{i}$ of pairwise disjoint sets such that $\lambda(X_{i})=\infty$ for every
$i\in\mathbb{N}$, since $(X_{i},\mathcal{L}_{X_{i}},\lambda_{X_{i}})$
satisfies the condition $\left(  \beta\right)  $ of Theorems \ref{t31} and \ref{t9}. Since by
the Proposition \ref{p2} we have $\operatorname{span}\left\{
{\textstyle\bigcup_{q\in(0,p)}}
L_{q}(X_{i})\cap L_{p}(X_{i})\right\}  =%
{\textstyle\bigcup_{q\in(0,p)}}
L_{q}(X_{i})\cap L_{p}(X_{i})$, the result follows from Theorem \ref{t29}.
\end{proof}

The following lemma gives us sufficient conditions for a measure space to
satisfy condition $\left(  \alpha\right)  $ of Theorems \ref{t31} and \ref{t9}. These
conditions will be crucial to show the main results of this section (see
Corollaries \ref{c38} e \ref{c39}).

\begin{lemma}
\label{bl32} Assume that $X$ is a separable metric space and let
$(X,\mathcal{M},\mu)$ be a measure space such that $\mathcal{M}$ contains the
$\sigma$-algebra of Borel. For each $x\in X$ and $r\in(0,\infty)$, let
$B(x,r)$ be the open ball of center $x$ and radius $r$. If there exists a
function $f\colon(0,\infty)\rightarrow(0,\infty)$ with $\lim
\limits_{x\rightarrow0}f(x)=0$ and $\mu(B(x,r))\leq f(r)$ for all $x\in X$ and
$r\in(0,\infty)$, then $(X,\mathcal{M},\mu)$ satisfies the condition ($\alpha
$) of Theorems \ref{t31} and \ref{t9}.
\end{lemma}

\begin{proof}
Let $\{x_{k}\}_{k\in\mathbb{N}}$ be a dense set in $X$. Given $\varepsilon>0$,
let $r>0$ such that $f(r)<\varepsilon$. Thus, since
\[
X=\bigcup_{k\in\mathbb{N}}\left(  X\cap B(x_{k},r)\right)
\]
and
\[
0<\mu(X)\leq\sum_{k=1}^{\infty}\mu(X\cap B(x_{k},r))\text{,}%
\]
there exists $k\in\mathbb{N}$ such that $\mu(X\cap B(x_{k},r))>0$. This
implies that $\mu(X\cap B(x_{k},r))\leq\mu(B(x_{k},r))\leq
f(r)<\varepsilon$. Therefore, we can conclude that $\inf\{\mu(A);\ A\in
\mathcal{M},\ \mu(A)>0\}=0$.
\end{proof}

Using the Principle of Dependent Choice, we obtain the following lemma, which
will be very useful to give sufficient conditions for a measure space to
satisfy the conditions of Theorem \ref{t29}. We will use this lemma to give a
characterization of a measurable subset of positive measure through
$(\alpha,\beta)$-spaceability (see Corollaries \ref{c38} e \ref{c39}).

\begin{lemma}
\label{bl33} Let $X$ be a separable metric space and $(X,\mathcal{M},\mu)$ be
a measure space. Assume that $\mathcal{M}$ contains the $\sigma$-algebra of
Borel. If there is a function $f:(0,\infty)\rightarrow(0,\infty)$ with
$\lim_{x\rightarrow0}f(x)=0$ and $\mu(B(x,r))\leq f(r)$ for all $x\in X$ and
$r\in(0,\infty)$, then there is a sequence $\{X_{k}\}_{k\in\mathbb{N}}$ of
measurable pairwise disjoint subsets of $X$ such that $\mu(X_{k})>0$ for
all $k\in\mathbb{N}$ and $X=\cup_{k\in\mathbb{N}}X_{k}$.
\end{lemma}

\begin{proof}
Consider the set
\[
A=\left\{  \mathcal{X}\subset\mathcal{P}(\mathcal{M}):
\begin{array}
[c]{l}%
1\leq\operatorname{card}(\mathcal{X})\leq\aleph_{0}; \ \ \ Y\cap Z=\emptyset \text{ if }
Y,Z\in\mathcal{X}, Y\neq Z;\\
\mu(Y)>0 \text{ if } Y\in\mathcal{X}; \ \ \ \mu\left(  X\smallsetminus
\textstyle\bigcup_{Y\in\mathcal{X}}Y\right)  >0
\end{array}
\right\}  .
\]
From Lemma \ref{bl32}, there is $Y\in\mathcal{M}$ such that $0<\mu
(Y)<\mu(X)$. Hence $\mu(X\smallsetminus Y)>0$, because $X=Y\cup
(X\smallsetminus Y)$ and otherwise, we have $\mu(X)=\mu(Y)+\mu
(X\smallsetminus Y)=\mu(Y)$, that is, $\{Y\}\in A$. Thus, $A$ is non-void.
Let us consider the following binary relation on $A$ given by
\[
R=\{(\mathcal{X},\mathcal{Y})\in A\times A;\ \mathcal{X}\subsetneq
\mathcal{Y}\}\text{.}%
\]
Given $\mathcal{X}\in A$, we have $\mu(X\smallsetminus%
{\textstyle\bigcup_{Y\in\mathcal{X}}}
Y)>0$. By Lemma \ref{bl32}, there is $Y_{\mathcal{X}}\in\mathcal{M}$ such that
$Y_{\mathcal{X}}\subset X\smallsetminus%
{\textstyle\bigcup_{Y\in\mathcal{X}}}
Y$ and $0<\mu(Y_{\mathcal{X}})<\mu(X\smallsetminus%
{\textstyle\bigcup_{Y\in\mathcal{X}}}
Y)$. This implies that
\begin{align*}
\mu\left(  X\smallsetminus%
{\textstyle\bigcup_{Y\in\mathcal{X}\cup\{Y_{\mathcal{X}}\}}}
Y\right)   &  =\mu\left(  X\smallsetminus\left(
{\textstyle\bigcup_{Y\in\mathcal{X}}}
Y\cup Y_{\mathcal{X}}\right)  \right) \\
&  =\mu\left(  \left(  X\smallsetminus%
{\textstyle\bigcup_{Y\in\mathcal{X}}}
Y\right)  \smallsetminus Y_{\mathcal{X}}\right)  >0,
\end{align*}
that is, $(\mathcal{X},\mathcal{X}\cup\{Y_{\mathcal{X}}\})\in R$. By the
Principle of Dependent Choice, we guarantee the existence of a sequence
$\{\mathcal{X}_{n}\}_{n\in\mathbb{N}}$ in $A$ such that $\mathcal{X}%
_{i}\subsetneq\mathcal{X}_{i+1}$, for each $i\in\mathbb{N}$. It is plain that
$\operatorname*{card}(\mathcal{X}_{n})\geq n$, for each $n\in\mathbb{N}$. This
assures that the set $B=%
{\textstyle\bigcup_{n\in\mathbb{N}}}
\mathcal{X}_{n}$ is countable. Let $\{\tilde{X}_{k}\}_{k\in\mathbb{N}}$ be an
enumeration of elements in $B$. If $n\neq m$, there are $r_{n},r_{m}%
\in\mathbb{N}$ such that $\tilde{X}_{n}\in\mathcal{X}_{r_{n}}$ and $\tilde
{X}_{m}\in\mathcal{X}_{r_{m}}$. Since the family $\{\mathcal{X}_{n}%
\}_{n\in\mathbb{N}}$ is fully ordered, we can assume without loss of
generality that $\mathcal{X}_{r_{n}}\subsetneq\mathcal{X}_{r_{m}}$, and thus
$\tilde{X}_{n},\tilde{X}_{m}\in\mathcal{X}_{r_{m}}$. Since $\mathcal{X}%
_{r_{m}}\in A$, we have $\mu(\tilde{X}_{n})>0$, $\mu(\tilde{X}_{m})>0$
and $\tilde{X}_{n}\cap\tilde{X}_{m}=\emptyset$. A sequence under the
conditions of the lemma is given by $X_{1}=\tilde{X}_{1}\cup\left(
X\smallsetminus%
{\textstyle\bigcup_{i\in\mathbb{N}}}
\tilde{X}_{i}\right)  $ and $X_{k}=\tilde{X}_{k}$, for $k>1$.
\end{proof}

Let us see a characterization of a measurable subset of positive measure
through $(\alpha,\beta)$-spaceability.

\begin{corollary}
\label{c38} Let $m\in\mathbb{N}$ and $p\in(0,\infty)$. Consider the Lebesgue
measure space $(\mathbb{R}^{m},\mathcal{L},\lambda)$ and let $X\in\mathcal{L}%
$. Then
\[
\lambda(X)>0\text{ if and only if }L_{p}(X)\smallsetminus%
{\textstyle\bigcup_{q\in(p,\infty]}}
L_{q}(X)\text{ is }(\alpha,\mathfrak{c})\text{-spaceable for all }%
\alpha<\aleph_{0}.
\]

\end{corollary}

\begin{proof}
If $\lambda(X)=0$, we have $L_{p}(X)=\{0\}$, for each $p\in(0,\infty]$. Thus,
$L_{p}(X)\smallsetminus%
{\textstyle\bigcup_{q\in(p,\infty]}}
L_{q}(X)$ is not $(n,\mathfrak{c})$-spaceable, for each $n\in\mathbb{N}$. On
the other hand, if $\lambda(X)>0$, we can use Lemma \ref{bl33} and
consider the partition $X=
{\textstyle\bigcup_{i\in\mathbb{N}}}
X_{i}$ with $\lambda(X_{i})>0$ for all $i\in\mathbb{N}$. It follows from
Proposition \ref{p2} that $\operatorname{span}(
{\textstyle\bigcup_{q\in(p,\infty]}}
L_{q}(X_{i})\cap L_{p}(X_{i}))=
{\textstyle\bigcup_{q\in(p,\infty]}}
L_{q}(X_{i})\cap L_{p}(X_{i})$, for each $i\in\mathbb{N}$. Now, Lemma
\ref{bl32} together Theorem \ref{t9}, allows us to conclude that $L_{p}%
(X_{i})\smallsetminus%
{\textstyle\bigcup_{q\in(p,\infty]}}
L_{q}(X_{i})$ is spaceable, for every $i\in\mathbb{N}$. Finally, the Theorem
\ref{t29} assures that $L_{p}(X)\smallsetminus%
{\textstyle\bigcup_{q\in(p,\infty]}}
L_{q}(X)$ is $(n,\mathfrak{c})$-spaceable, for each $n\in\mathbb{N}$.
\end{proof}

Combining the Corollary \ref{c38} with \cite[Corollary 2.3]{fprs}, we also
obtain the following characterization:

\begin{corollary}
\label{c39} Let $m\in\mathbb{N}$, $p\in(0,\infty)$ and $X\subset\mathbb{R}%
^{m}$ be a measurable subset with $\lambda(X)>0$. The set $L_{p}(X)\smallsetminus%
{\textstyle\bigcup_{q\in(p,\infty]}}
L_{q}(X)$ is $(\alpha,\mathfrak{c})$-spaceable in $L_{p}(X)$ if and only if
$\alpha<\aleph_{0}$.
\end{corollary}

\begin{remark}
Corollaries \ref{c38} and \ref{c39} are also true if we take $(X,\mathcal{M}
,\mu)$ as in Lemmas \ref{bl32} and \ref{bl33}, instead of $(\mathbb{R}
^{m},\mathcal{L},\lambda)$.
\end{remark}

In summary, we have:

\begin{proposition}
Let $m\in\mathbb{N}$ and $p\in(0,\infty)$. Consider the Lebesgue measure space
$(\mathbb{R}^{m},\mathcal{L},\lambda)$ and let $X\in\mathcal{L}$. The
following statements are equivalent:

\begin{itemize}
\item[(a)] $\lambda(X)>0$;

\item[(b)] $L_{p}(X)\smallsetminus%
{\textstyle\bigcup_{q\in(p,\infty]}}
L_{q}(X)$ is $(n,\mathfrak{c})$-spaceable for some $n<\aleph_{0}$;

\item[(c)] $L_{p}(X)\smallsetminus%
{\textstyle\bigcup_{q\in(p,\infty]}}
L_{q}(X)$ is $(\alpha,\mathfrak{c})$-spaceable for all $\alpha<\aleph_{0}$;

\item[(d)] $L_{p}(X)\smallsetminus%
{\textstyle\bigcup_{q\in(p,\infty]}}
L_{q}(X)$ is $\mathfrak{c}$-spaceable;

\item[(e)] $L_{p}(X)\smallsetminus%
{\textstyle\bigcup_{q\in(p,\infty]}}
L_{q}(X)$ is $\mathfrak{c}$-lineable.
\end{itemize}
\end{proposition}

In \cite{Bernal} the authors showed that the set $\left(  \mathcal{C}%
_{0}\left(  I\right)  \cap R\left(  I\right)  \right)  \smallsetminus
{\bigcup_{p\in(0,\infty)}}L_{p}\left(  I\right)  $, is spaceable in
$\mathcal{C}_{0}\left(  I\right)  \cap R\left(  I\right)  $, under the norm
\[
\Vert f\Vert=\underset{x\geq0}{\sup}\left\vert f\left(  x\right)  \right\vert
+\underset{x\geq0}{\sup}\left\vert \int_{0}^{x}f\left(  t\right)
\,dt\right\vert \text{.}%
\]
Now we complement this result through the following corollary:

\begin{corollary}
The set
\[
\left(  \mathcal{C}_{0}\left(  I\right)  \cap R\left(  I\right)  \right)
\smallsetminus{\bigcup_{p\in(0,\infty)}}L_{p}\left(  I\right)
\]
is $(\alpha,\mathfrak{c})$-spaceable if and only if $\alpha<\aleph_{0}$.
\end{corollary}

\begin{proof}
Considering $V=\mathcal{C}_{0}\left(  I\right)  \cap R\left(  I\right)  $ and
$W_{r}=\mathcal{C}_{0}\left(  I\right)  \cap R\left(  I\right)  \cap
L_{r}\left(  I\right)  $, $\left(  r\in\mathbb{N}\right)  $, we have
$\mathcal{C}_{0}\left(  I\right)  \cap L_{p}\left(  I\right)  \subset
\mathcal{C}_{0}\left(  I\right)  \cap L_{q}\left(  I\right)  $, for $p\leq q$.
Hence, the set $\bigcup_{p\in(0,\infty)}L_{p}\left(  I\right)  \cap
\mathcal{C}_{0}\left(  I\right)  \cap R\left(  I\right)  $ is a vector
subspace of $\mathcal{C}_{0}\left(  I\right)  \cap R\left(  I\right)  $. For
every $r\in\mathbb{N}$, since $\mathcal{C}_{0}\left(  I\right)  \cap R\left(
I\right)  \cap L_{r}\left(  I\right)  $ is a Banach space under the norm
\[
\Vert f\Vert=\underset{x\geq0}{\sup}\left\vert f\left(  x\right)  \right\vert
+\underset{x\geq0}{\sup}\left\vert \int_{0}^{x}f\left(  t\right)
\,dt\right\vert + \left(  \int_{0}^{\infty}|f(x)|^{r}dx\right)  ^{\frac{1}%
{r}}
\]
the result follows from Corollary \ref{Corollary 2.7}.
\end{proof}

\subsection{Sequence spaces}

Let $X$ be a Banach space over $\mathbb{K}$. As usual, $X^{\mathbb{N}}$ will
denote the space of all $X$-valued sequences endowed with the natural
operations of vector sum and scalar multiplication. Now we remember the
definitions of some subspaces of $X^{\mathbb{N}}$ that will be investigated in
this section:

\begin{itemize}
\item If $0< p<\infty$, then
\[
\ell_{p}\left(  X\right)  :=\left\{  \left(  x_{n}\right)  _{n=1}^{\infty}\in
X^{\mathbb{N}}:\sum_{n=1}^{\infty}\left\Vert x_{n}\right\Vert ^{p}%
<\infty\right\}  \text{.}%
\]
is the space of all $p$-summable sequences in $X^{\mathbb{N}}$. The function
$\left\Vert \cdot\right\Vert _{p}\colon\ell_{p}\left(  X\right)
\rightarrow\mathbb{R}$ given by%
\[
\left\Vert \left(  x_{n}\right)  _{n=1}^{\infty}\right\Vert _{p}=\left(
\sum_{n=1}^{\infty}\left\Vert x_{n}\right\Vert ^{p}\right)  ^{1/p}%
\]
defines a norm over $\ell_{p}\left(  X\right)  $ when $1\leq p<\infty$ and
defines a $p$-norm when $0<p<1$. In the first situation $\ell_{p}\left(
X\right)  $ is a Banach space and, in the sequence, a $p$-Banach space. When
$X=\mathbb{K}$, as usual, we write $\ell_{p}$ instead $\ell_{p}\left(
\mathbb{K}\right)  $.

\item The Banach space of all bounded sequences $\left(  x_{n}\right)
_{n=1}^{\infty}\in X^{\mathbb{N}}$ with the norm $\left\Vert \cdot\right\Vert
_{\infty}$ defined by%
\[
\left\Vert \left(  x_{n}\right)  _{n=1}^{\infty}\right\Vert _{\infty}%
=\sup_{n\in\mathbb{N}}\left\Vert x_{n}\right\Vert
\]
will be denoted by $\ell_{\infty}\left(  X\right)  $. When $X=\mathbb{K}$, we
write $\ell_{\infty}$ instead $\ell_{\infty}\left(  \mathbb{K}\right)  $.

\item The set $c(X)$ of all convergent sequences in $X^{\mathbb{N}}$ is a
closed subspace of $\ell_{\infty}(X)$ with the norm $\left\Vert \cdot
\right\Vert _{\infty}$. When $X=\mathbb{K}$, we write $c$ instead $c\left(
\mathbb{K}\right)  $.

\item The Banach space of all sequences $\left(  x_{n}\right)  _{n=1}^{\infty
}\in X^{\mathbb{N}}$ such that $\lim x_{n}=0$ with the norm $\left\Vert
\cdot\right\Vert _{\infty}$ will be denoted by $c_{0}\left(  X\right)  $. When
$X=\mathbb{K}$, we write $c_{0}$ instead $c_{0}\left(  \mathbb{K}\right)  $.

\item The space of all eventually null sequences in $X^{\mathbb{N}}$, that we
denote by $c_{00}\left(  X\right)  $ ($c_{00}$, when $X=\mathbb{K}$), is the
space of all sequence $\left(  x_{n}\right)  _{n=1}^{\infty}$ for which there
is a positive integer $n_{0}$ such that $x_{n}=0$ whenever $n\geq n_{0}$.

\item Let $X^{\ast}$ be the topological dual of $X$, that is, the space of all
continuous linear functional over $X$, and let $1\leq p<\infty$. Then $\ell
_{p}^{w}\left(  X\right)  $ will denote the Banach space of all sequences
$\left(  x_{n}\right)  _{n=1}^{\infty}\in X^{\mathbb{N}}$ such that $\left(
\varphi\left(  x_{n}\right)  \right)  _{n=1}^{\infty}\in\ell_{p}$ for every
$\varphi\in X^{\ast}$ with the norm $\left\Vert \cdot\right\Vert _{p,w}$
defined by%
\[
\left\Vert \left(  x_{n}\right)  _{n=1}^{\infty}\right\Vert _{p,w}%
=\sup_{\varphi\in B_{X^{\ast}}}\left\Vert \left(  \varphi\left(  x_{n}\right)
\right)  _{n=1}^{\infty}\right\Vert _{p}\text{,}%
\]
where $B_{X^{\ast}}$ is the closed unit ball of $X^{\ast}$. When
$X=\mathbb{K}$, we write $\ell_{p}^{w}$ instead $\ell_{p}^{w}\left(
\mathbb{K}\right)  $.

\item The space of all sequences $\left(  x_{n}\right)  _{n=1}^{\infty}\in
\ell_{p}^{w}\left(  X\right)  $ such that
\[
\lim_{k\rightarrow\infty}\left\Vert \left(  x_{n}\right)  _{n=k}^{\infty
}\right\Vert _{p,w}=0
\]
will be denoted by $\ell_{p}^{u}\left(  X\right)  $. When endowed with the
norm $\left\Vert \cdot\right\Vert _{p,w}$, $\ell_{p}^{u}\left(  X\right)  $
becomes a closed subspace of $\ell_{p}^{w}\left(  X\right)  $. It is usual to
say that $\ell_{p}^{u}\left(  X\right)  $ is the space of unconditional
$p$-summable sequences in $X^{\mathbb{N}}$.

\item For $1\leq p<\infty$, let $\ell_{p}\left\langle X\right\rangle $ the
Banach space of all Cohen $p$-summable sequences $\left(  x_{n}\right)
_{n=1}^{\infty}\in X^{\mathbb{N}}$, that is, the space of all $X$-valued
sequences such that%
\[
\left\Vert \left(  x_{n}\right)  _{n=1}^{\infty}\right\Vert _{\ell
_{p}\left\langle X\right\rangle }:=\sup\left\{  \sum_{n=1}^{\infty}\left\vert
y_{n}^{\ast}\left(  x_{n}\right)  \right\vert :\left(  y_{n}\right)
_{n=1}^{\infty}\in\ell_{p}^{w}\left(  X^{\ast}\right)  \text{ and }\left\Vert
\left(  y_{n}^{\ast}\right)  _{n=1}^{\infty}\right\Vert _{p^{\ast}%
,w}=1\right\}  <\infty\text{,}%
\]
where $p^{\ast}$ is the Lebesgue conjugate of $p$. If $X=\mathbb{K}$, then
$\ell_{p}\left\langle \mathbb{K}\right\rangle =\ell_{p}$.

\item Given $0<p\leq s\leq\infty$, by $\ell_{m(s;p)}(X)$ we mean the Banach
($p$-Banach if $0<p<1$) space of all mixed $(s,p)$-summable sequences on $X$
(see \cite{BDFP}), i.e., the space of all sequences $(x_n)_{n=1}^{\infty} \in X^{\mathbb{N}}$ such that $x_n=\lambda_ny_n$ for all $n\in\mathbb{N}$, with $(\lambda_n)_{n=1}^\infty\in \ell_{s^{\ast}(q)}$ and $(y_n)_{n=1}^\infty\in \ell_s^w(X)$, where $s^{\ast}(q)$ is the $q$-Lebesgue conjugate of $s$, that is, $1/s+1/s^{\ast}(q)=1/q$.


\end{itemize}

A first application in this context is the following result, which is an
immediate consequence of Corollary \ref{coro1}:

\begin{corollary}
Let $X=c,c_{0}$ or $\ell_{p}$, $p\in[1,\infty]$. The set $X\smallsetminus
c_{00}$ is $(\alpha,\mathfrak{c})$-spaceable in $X$ if and only if $\alpha<\aleph
_{0}$.
\end{corollary}

The following result is already known in the literature (see \cite{FPT}).
Below, as a consequence of Theorems \ref{Theorem 2.4} and \ref{t29}, we show
how we can recover such a result.

\begin{corollary}
Let $p\in(0,\infty)$. The set $\ell_{p}\smallsetminus%
{\textstyle\bigcup_{0<q<p}}
\ell_{q}$ is $(\alpha,\mathfrak{c})$-spaceable if and only if $\alpha
<\aleph_{0}$.
\end{corollary}

\begin{proof}
The case $p\geq1$ is a straightforward consequence of Corollary
\ref{Corollary 2.7}. For the case $0<p<1$, it is enough to remember that
$\ell_{r}=\ell_{r}(\mathbb{N})=L_{r}(\mathbb{N},\mathcal{P}(\mathbb{N}),\mu)$
for each $r>0$, where the measure $\mu$ is given by $\mu
(A)=\operatorname*{card}(A)$ whenever $\operatorname*{card}(A)<\aleph_{0}$ and
$\mu(A)=\infty$ otherwise. It is well-known the existence of a sequence
$\{\mathbb{N}_{i}\}_{i\in\mathbb{N}}$ of pairwise disjoint sets such that
$\operatorname*{card}(\mathbb{N}_{i})=\aleph_{0}$ for each $i\in\mathbb{N}$
and $\mathbb{N}=%
{\textstyle\bigcup_{i\in\mathbb{N}}}
\mathbb{N}_{i}$. Furthermore, for each $i\in\mathbb{N}$,
\[
L_{r}(\mathbb{N}_{i},\mathcal{P}(\mathbb{N})_{\mathbb{N}_{i}},\mu
_{\mathbb{N}_{i}})=L_{r}(\mathbb{N}_{i},\mathcal{P}(\mathbb{N}_{i}%
),\mu_{\mathbb{N}_{i}})
\]
is naturally isometrically isomorphic to $L_{r}(\mathbb{N},\mathcal{P}%
(\mathbb{N}),\mu)$ for every $r>0$, which implies that we have the lineability of
\[
L_{p}(\mathbb{N}_{i},\mathcal{P}(\mathbb{N})_{\mathbb{N}_{i}},\mu
_{\mathbb{N}_{i}})\smallsetminus%
{\textstyle\bigcup_{0<q<p}}
L_{q}(\mathbb{N}_{i},\mathcal{P}(\mathbb{N})_{\mathbb{N}_{i}},\mu
_{\mathbb{N}_{i}})\text{.}%
\]
Therefore, as a consequence of Theorem \ref{t29} we can infer that
\[
\ell_{p}\smallsetminus%
{\textstyle\bigcup_{0<q<p}}
\ell_{q}=L_{p}(\mathbb{N},\mathcal{P}(\mathbb{N}),\mu)\smallsetminus%
{\textstyle\bigcup_{0<q<p}}
L_{q}(\mathbb{N},\mathcal{P}(\mathbb{N}),\mu)
\]
is $(\alpha,\mathfrak{c})$-spaceable for every $\alpha<\aleph_{0}$.
\end{proof}

In \cite{bbfp,BCP,BDFP} the authors use different techniques for checking
spaceability in different types of sequence spaces. Here, our goal now is to
investigate such results in the framework of $\left(  \alpha,\beta\right)  $-spaceability.

\begin{lemma}
\label{t12}Let $X$ be a normed space. If $E,F\subset X^{\mathbb{N}}$ are
Banach spaces carrying the norms $\Vert\cdot\Vert_{E}$ and $\Vert\cdot
\Vert_{F}$, respectively, and norm convergence in these spaces implies
coordinatewise convergence in $X$, then the space $E\cap F$ is a Banach space
with the norm $\Vert\cdot\Vert_{E}+\Vert\cdot\Vert_{F}$.
\end{lemma}

\begin{proof}
Let $\{(x_{n}^{(k)})_{n\in\mathbb{N}}\}_{k\in\mathbb{N}}$ be a Cauchy sequence
in $(E\cap F,\Vert\cdot\Vert_{E}+\Vert\cdot\Vert_{F})$. Then this sequence is
a Cauchy sequence both in $(E,\Vert\cdot\Vert_{E})$ and $(F,\Vert\cdot
\Vert_{F})$. Since $(E,\Vert\cdot\Vert_{E})$ and $(F,\Vert\cdot\Vert_{F})$ are
Banach spaces, there are sequences $(x_{n})_{n\in\mathbb{N}}$ and
$(y_{n})_{n\in\mathbb{N}}$ such that $\Vert(x_{n}^{(k)})_{n\in\mathbb{N}%
}-(x_{n})_{n\in\mathbb{N}}\Vert_{E}\overset{n\rightarrow\infty}%
{\longrightarrow}0$ and $\Vert(x_{n}^{(k)})_{n\in\mathbb{N}}-(y_{n}%
)_{n\in\mathbb{N}}\Vert_{F}\overset{n\rightarrow\infty}{\longrightarrow}0$.
Since norm convergence in this spaces implies coordinatewise convergence in
$X$, we can conclude that $(x_{n})_{n\in\mathbb{N}}=(y_{n})_{n\in\mathbb{N}%
}\in E\cap F$ and $(x_{n}^{(k)})_{n\in\mathbb{N}}\overset{n\rightarrow\infty
}{\longrightarrow}(x_{n})_{n\in\mathbb{N}}$ in $(E\cap F,\Vert\cdot\Vert
_{E}+\Vert\cdot\Vert_{F})$.
\end{proof}

Let $\left(  X_{n}\right)  _{n=1}^{\infty}$ be a sequence of Banach spaces over $\mathbb{K}$. Given $0<p<\infty$, by $\left(\sum_{n}X_{n}\right)_{p}$ we mean the vector space of of all sequences $\left(  x_{n}\right)_{n=1}^{\infty}$ such that $x_{n}\in X_{n}$ for every $n$ and $\left(\left\Vert x_{n}\right\Vert _{X_{n}}\right)  _{n=1}^{\infty}\in\ell_{p}$. Making the obvious modification for $p = 0$, we define $\left(\sum_{n}X_{n}\right)_{0}$ as the vector space of of all sequences $\left(x_{n}\right)_{n=1}^{\infty}$ such that $x_{n}\in X_{n}$ for every $n$ and $\left(\left\Vert x_{n}\right\Vert _{X_{n}}\right)_{n=1}^{\infty}\in c_{0}$ (see \cite{bbfp}).

Moreover, given a Banach space $X$, a family $(X_n)_{n=1}^{\infty}$ of Banach spaces is said to contain isomorphs of $X$ uniformly if there are $\delta>0$ and a family of isomorphisms into $R_n : X \to X_n$ such that $\min\{\|R_n\|,\|R_n^{-1}\|\} \leq\delta$ for every $n \in\mathbb{N}$.

\begin{corollary}
\label{c22}Let $(X_{n})_{n\in\mathbb{N}}$ be a sequence of Banach spaces that
contains a subsequence containing isomorphs of the infinite dimensional Banach
space $X$ uniformly. Then

\begin{itemize}
\item[(a)] for $1<p<\infty$, $\left(  \sum_{n}X_{n}\right)  _{p}\smallsetminus\textstyle\bigcup_{1\leq q<p}\left(  \sum_{n}X_{n}\right)  _{q}$ is $(\alpha,\mathfrak{c})$-spaceable if and only if $\alpha<\aleph_{0}$;
\item[(b)] $\left(  \sum_{n}X_{n}\right)  _{0}\smallsetminus%
{\textstyle\bigcup_{1\leq q}}
\left(  \sum_{n}X_{n}\right)  _{q}$ is $(\alpha,\mathfrak{c})$-spaceable if and only if $\alpha<\aleph_{0}$.
\end{itemize}
\end{corollary}

\begin{proof}
\begin{itemize}

\item[(a)] Notice that $\left\{  \left(  \sum_{n}X_{n}\right)  _{q}\right\}_{q\in\lbrack1,p)}$ is totally ordered, and therefore $\textstyle\bigcup_{1\leq
q<p}\left(  \sum_{n}X_{n}\right)  _{q}$ is a linear subspace of $\left(
\sum_{n}X_{n}\right)  _{p}$. Let $(q_{n})_{n\in\mathbb{N}}$ a sequence of
positive integers such that
\begin{equation}
\bigcup_{1\leq q<p}\left(  \sum_{n}X_{n}\right)  _{q}=\bigcup_{n\in\mathbb{N}%
}\left(  \sum_{n}X_{n}\right)  _{q_{n}}. \label{ber}%
\end{equation}
Hence we are in the conditions of Corollary \ref{Corollary 2.7} since $\Vert
\cdot\Vert_{p}\leq\Vert\cdot\Vert_{q}$ for $q<p$ and \cite[Theorem 3.2]{bbfp}
assures that (\ref{ber}) has infinite codimension.

\item[(b)] Analogous to item (a).
\end{itemize}
\end{proof}

The following results are generalizations of the results obtained in
\cite{BCP,BDFP} and their proofs follow similar steps to Corollary \ref{c22}
together with Lemma \ref{t12}.

\begin{corollary}
Let $X$ be an infinite dimensional Banach space. Then
\[
c_{0}(X)\smallsetminus\bigcup_{1\leq p}\ell_{p}^{w}(X)
\]
is $(\alpha,\mathfrak{c})$-spaceable if and only if $\alpha<\aleph_{0}$.
\end{corollary}

\begin{corollary}
Let $1<p<\infty$ and $X$ be an infinite dimensional Banach space. Then
\[
\ell_{p}(X)\smallsetminus\bigcup_{1<q<p}\ell_{q}^{w}(X)\text{,}\ \ \ \ \ \ell
_{p}^{u}(X)\smallsetminus\bigcup_{1<q<p}\ell_{q}^{w}(X)\ \ \ \ \ \text{and}%
\ \ \ \ \ \ell_{p}\left\langle X\right\rangle \smallsetminus\bigcup
_{1<q<p}\ell_{q}^{w}(X)
\]
are $(\alpha,\mathfrak{c})$-spaceable if and only if $\alpha<\aleph_{0}$. In
particular, $\ell_{p}^{w}(X)\smallsetminus%
{\textstyle\bigcup_{1<q<p}}
\ell_{q}^{w}(X)$ is $(\alpha,\mathfrak{c})$-spaceable if and only if $\alpha<\aleph_{0}$.
\end{corollary}

\begin{corollary}
The sets $\ell_{m(s;p)}(X)\smallsetminus\ell_{p}(X)$ and $\ell_{p}^{u}(X)\smallsetminus
\ell_{p}(X)$ are $(\alpha,\mathfrak{c})$-spaceable for every $\alpha
<\aleph_{0}$, $1\leq p\leq s<\infty$ and every infinite dimensional Banach
space $X$. In particular $\ell_{p}^{w}(X)\smallsetminus\ell_{p}(X)$ is
$(\alpha,\mathfrak{c})$-spaceable if and only if $\alpha<\aleph_{0}$.
\end{corollary}

\subsection{H\"{o}lder spaces}

Let $X$ and $Y$ be metric spaces. For $\alpha\in\lbrack0,1]$, we define the
H\"{o}lder's seminorm of a function $f\colon X\longrightarrow Y$ by
\[
\lbrack f]_{\alpha}=\sup\limits_{x\neq y}\frac{d_{Y}(f(x),f(y))}%
{d_{X}(x,y)^{\alpha}}\text{.}%
\]
Let $E$ be a Banach space. For $\alpha\in\lbrack0,1]$ we define the
H\"{o}lder's norm of order $\alpha$ of a function $f\colon X\longrightarrow E$
by
\[
\Vert f\Vert_{\mathcal{C}^{0,\alpha}}=\Vert f\Vert_{\mathcal{C}^{0}%
}+[f]_{\alpha},
\]
where $\Vert f\Vert_{\mathcal{C}^{0}}=\sup\limits_{x\in X}\Vert f(x)\Vert$. We
define the H\"{o}lder space of order $\alpha$ by
\[
\mathcal{C}^{0,\alpha}(X,E)=\{f\colon X\longrightarrow E;\ \Vert
f\Vert_{\mathcal{C}^{0,\alpha}}<\infty\}\text{.}%
\]
When $E=\mathbb{K}$, we will denote $\mathcal{C}^{0,\alpha}(X,E)$ by
$\mathcal{C}^{0,\alpha}(X)$.

In particular, when $\alpha\in(0,1]$, every function in $\mathcal{C}%
^{0,\alpha}(X,E)$ is continuous, while $\mathcal{C}^{0,0}(X,E)$ is the set of
all bounded functions $f\colon X\longrightarrow E$.

\begin{lemma}
\label{l6} For each $\alpha\in\lbrack0,1]$, $(\mathcal{C}^{0,\alpha
}(X,E),\Vert\cdot\Vert_{\mathcal{C}^{0,\alpha}})$ is a Banach space.
\end{lemma}

It is well-known that, for $0\leq\alpha<\beta\leq1$,
\[
\mathcal{C}^{0,\beta}(X,E)\subset\mathcal{C}^{0,\alpha}(X,E)\text{.}%
\]


The following result will\ be very useful for our applications in this framework.

\begin{theorem}
\label{t10}If $X$ is compact and $E$ is finite dimensional, then, for all
$0\leq\alpha<\beta\leq1$ the natural embedding%
\[
\mathcal{C}^{0,\beta}(X,E)\hookrightarrow{C}^{0,\alpha}(X,E)
\]
is compact.
\end{theorem}

\begin{proof}
Let $(f_{n})_{n\in\mathbb{N}}$ be a bounded sequence in $\mathcal{C}^{0,\beta
}(X,E)$. The boundedness of the functions in $\mathcal{C}^{0,\beta}(X,E)$
provides, simultaneously, equicontinuity and uniform boundedness for $(f_n)_{n\in\mathbb{N}}$. Hence, by
Arzela-Ascoli Theorem, up to a subsequence, we can assume that $(f_{n}%
)_{n\in\mathbb{N}}$ converges uniformly for certain $f_{\infty}$. Since%
\[
\frac{\Vert f_{\infty}(x)-f_{\infty}(y)\Vert}{d_{X}(x,y)^{\beta}}%
=\lim\limits_{n\rightarrow\infty}\frac{\Vert f_{n}(x)-f_{n}(y)\Vert}%
{d_{X}(x,y)^{\beta}}\leq\limsup\limits_{n\rightarrow\infty}[f_{n}]_{\beta
}<\infty
\]
for any $x\neq y$ in $X$, and, obviously, we have $\Vert f_{\infty}%
\Vert_{\mathcal{C}^{0}}<\infty$, it follows that $f_{\infty}\in\mathcal{C}%
^{0,\beta}(X,E)\subset\mathcal{C}^{0,\alpha}(X,E)$. Denoting $g_{n}%
=f_{n}-f_{\infty}$, we have $\Vert g_{n}\Vert_{\mathcal{C}^{0}}\overset
{n\rightarrow\infty}{\longrightarrow}0$. It remains only to verify that
$[g_{n}]_{\alpha}\overset{n\rightarrow\infty}{\longrightarrow}0$. Notice that
for $x,y\in X$, with $x\neq y$, we have%
\[
\frac{\Vert g_{n}(x)-g_{n}(y)\Vert}{d_{X}(x,y)^{\alpha}}=\left(  \frac{\Vert
g_{n}(x)-g_{n}(y)\Vert}{d_{X}(x,y)^{\beta}}\right)  ^{\frac{\alpha}{\beta}%
}\cdot\Vert g_{n}(x)-g_{n}(y)\Vert^{1-\frac{\alpha}{\beta}}\leq\lbrack
g_{n}]_{\beta}^{\frac{\alpha}{\beta}}\left(  2\Vert g_{n}\Vert_{\mathcal{C}%
^{0}}\right)  ^{1-\frac{\alpha}{\beta}}\text{,}%
\]
that is, $[g_{n}]_{\alpha}\leq\lbrack g_{n}]_{\beta}^{\frac{\alpha}{\beta}%
}\left(  2\Vert g_{n}\Vert_{\mathcal{C}^{0}}\right)  ^{1-\frac{\alpha}{\beta}%
}$. Since the sequence $([g_{n}]_{\beta})_{n\in\mathbb{N}}$ is bounded, the
result follows.
\end{proof}

The following proposition will be\ useful for our first application in the
setting of Banach spaces.

\begin{proposition}
\label{p4}Let $(V,\Vert\cdot\Vert_{V})$ be an infinite dimensional Banach
space and let $W$ be a linear subspace of $V$. If there exists a norm
$\Vert\cdot\Vert_{W}$ that makes $W$ a Banach space and makes the inclusion
$W\hookrightarrow V$ compact, then $W$ has infinite codimension.
\end{proposition}

\begin{proof}
Assume that $W$ has finite codimension and let $U$ be a finite dimensional linear subspace of $V$
such that $W\oplus U=V$. Obviously, $W\oplus U$ is complete with the norm
$\Vert\cdot\Vert_{W}+\Vert\cdot\Vert_{V}$. If $(w_{n}+u_{n})_{n\in\mathbb{N}}$
is a bounded sequence in $(W\oplus U,\Vert\cdot\Vert_{W}+\Vert\cdot\Vert_{V}%
)$, then the sequence $(w_{n})_{n\in\mathbb{N}}$ admits a subsequence
$(w_{n})_{n\in\mathbb{N}_{1}}$ ($\mathbb{N}_{1}\subset\mathbb{N}$) that
converges in $(V,\Vert\cdot\Vert_{V})$ thanks to the compactness of
$W\hookrightarrow V$, and $(u_{n})_{n\in\mathbb{N}_{1}}$ admits a subsequence
$(u_{n})_{n\in\mathbb{N}_{2}}$ ($\mathbb{N}_{2}\subset\mathbb{N}_{1}$) that
converges in $(V,\Vert\cdot\Vert_{V})$ since $U$ is finite dimensional. Hence,
we infer that $(w_{n}+u_{n})_{n\in\mathbb{N}}$ admits a convergent subsequence
in $(V,\Vert\cdot\Vert_{V})$, and this yields\ that the inclusion%
\begin{equation}
(W\oplus U,\Vert\cdot\Vert_{W}+\Vert\cdot\Vert_{V})\hookrightarrow
(V,\Vert\cdot\Vert_{V}) \label{equa2}%
\end{equation}
is compact. Since the inclusion (\ref{equa2}) is also an isomorphism, by the
Open Mapping Theorem, we conclude that the unit ball of $V$ is compact, and
this is a contradiction since $V$ is infinite dimensional.
\end{proof}

A joint application of Proposition \ref{p4} and Corollary \ref{Corollary 2.7} yield the following result:

\begin{corollary}
Let $(V,\Vert\cdot\Vert_{V})$ be an infinite dimensional Banach space and let
$W$ be a linear subspace of $V$. If there exists a norm $\Vert\cdot\Vert_{W}$
that makes $W$ a Banach space and makes the inclusion $W\hookrightarrow V$
compact, then
\[
V\smallsetminus W\text{ is }\left(  \alpha,\mathfrak{c}\right)
\text{-spaceable if and only if }\alpha<\aleph_{0}\text{.}%
\]

\end{corollary}

Since $\mathcal{C}^{1}[0,1]\subset\mathcal{C}^{0,1}[0,1]\subset\mathcal{C}%
^{0,\alpha}[0,1]$ for each $\alpha\in\lbrack0,1]$, we conclude that
$\mathcal{C}^{0,\alpha}[0,1]$ is an infinite dimensional Banach space.

\begin{corollary}
If $\alpha,\beta\in\lbrack0,1]$ are such that $\alpha<\beta$, then
$\mathcal{C}^{0,\alpha}[0,1]\smallsetminus\mathcal{C}^{0,\beta}[0,1]$ is
$\left(  \alpha,\mathfrak{c}\right)  $-spaceable if and only if
$\alpha<\aleph_{0}$.
\end{corollary}


\begin{proposition}
\label{p45}For each $\alpha,\lambda\in(0,1)$ the function%
\[
f_{\lambda,\alpha}(x)=%
\begin{cases}
0\text{,} & \text{if}\ x\leq\lambda,\\
(x-\lambda)^{\alpha}\text{,} & \text{if}\ x>\lambda
\end{cases}
\]
belongs to $\mathcal{C}^{0,\alpha}[0,1]\smallsetminus%
{\textstyle\bigcup_{\alpha<\beta}}
\mathcal{C}^{0,\beta}[0,1]$.
\end{proposition}

\begin{proof}
Let $y<x$ in $[0,1]$. If $x,y\in\lbrack0,\lambda]$, we have%
\[
\frac{|f_{\lambda,\alpha}(x)-f_{\lambda,\alpha}(y)|}{|x-y|^{\alpha}}=0\text{.}%
\]
If $x,y\in\lbrack\lambda,1]$, we have
\begin{align*}
|f_{\lambda,\alpha}(x)-f_{\lambda,\alpha}(y)|  &  =(x-\lambda)^{\alpha
}-(y-\lambda)^{\alpha}\\
&  =\int_{y-\lambda}^{x-\lambda}\alpha t^{\alpha-1}dt\\
&  \leq\int_{y-\lambda}^{x-\lambda}\alpha(t-y+\lambda)^{\alpha-1}%
dt=(x-y)^{\alpha}=|x-y|^{\alpha}\text{.}%
\end{align*}
If $x\in(\lambda,1]$ and $y\in\lbrack0,\lambda)$ we have
\[
\frac{|f_{\lambda,\alpha}(x)-f_{\lambda,\alpha}(y)|}{|x-y|^{\alpha}}%
=\frac{(x-\lambda)^{\alpha}}{|x-y|^{\alpha}}<1\text{.}%
\]
Therefore, we can conclude that $[f_{\lambda,\alpha}]_{\alpha}=1$, that is,
$f_{\lambda,\alpha}\in\mathcal{C}^{0,\alpha}[0,1]$. On the other hand, for
$\beta>\alpha$ we have%
\[
\lbrack f_{\lambda,\alpha}]_{\beta}\geq\lim\limits_{t\rightarrow0^{+}}%
\frac{|f_{\lambda,\alpha}(\lambda+t)-f_{\lambda,\alpha}(\lambda-t)|}%
{|2t|^{\beta}}=\lim\limits_{t\rightarrow0^{+}}\frac{t^{\alpha}}{|2t|^{\beta}%
}=\lim\limits_{t\rightarrow0^{+}}\frac{1}{2^{\beta}\cdot t^{\beta-\alpha}%
}=\infty\text{.}%
\]
In other words, $f\notin \mathcal{C}^{0,\beta}[0,1]$, as required.
\end{proof}

\begin{theorem}
\label{t11}Given $\alpha,\lambda\in(0,1)$, let $f_{\lambda,\alpha}$ be as in
Proposition \ref{p45}. Then the set $\left\{  f_{\lambda,\alpha}:\lambda
\in(0,1)\right\}  $ is linearly independent and
\[
\operatorname{span}\left\{  f_{\lambda,\alpha}:\lambda\in(0,1)\right\}
\subset\left(  \mathcal{C}^{0,\alpha}[0,1]\smallsetminus%
{\textstyle\bigcup_{\alpha<\beta}}
\mathcal{C}^{0,\beta}[0,1]\right)  \cup\{0\}\text{.}%
\]
Hence $\mathcal{C}^{0,\alpha}[0,1]\smallsetminus\textstyle\bigcup_{\alpha<\beta}\mathcal{C}^{0,\beta}[0,1]$ is $\mathfrak{c}$-lineable.
\end{theorem}

\begin{proof}
Given $\beta>\alpha$, let $a_{1},\ldots,a_{n}\in\mathbb{R}\smallsetminus\{0\}$
and $\lambda_{1},\ldots,\lambda_{n}\in(0,1)$ with $\lambda_{1}<\ldots
<\lambda_{n}$. Then we have
\begin{align*}
\left[  \sum_{i=1}^{n}a_{i}f_{\lambda_{i},\alpha}\right]  _{\beta}  &
\geq\lim_{t\rightarrow0^{+}}\frac{\left\vert \sum_{i=1}^{n}a_{i}f_{\lambda
_{i},\alpha}(\lambda_{1}+t)-\sum_{i=1}^{n}a_{i}f_{\lambda_{i},\alpha}%
(\lambda_{1}-t)\right\vert }{|2t|^{\beta}}\\
&  =\lim_{t\rightarrow0^{+}}\frac{\left\vert a_{1}f_{\lambda_{1},\alpha
}(\lambda_{1}+t)\right\vert }{|2t|^{\beta}}\\
&  =\lim_{t\rightarrow0^{+}}\frac{\left\vert a_{1}\right\vert }{2^{\beta
}t^{\beta-\alpha}}=\infty\text{.}%
\end{align*}
The previous calculation gives us,\ simultaneously, that $\{f_{\lambda,\alpha
}:\ \lambda\in(0,1)\}$ is linearly independent and
\[
\operatorname{span}\{f_{\lambda,\alpha}:\ \lambda\in(0,1)\}\subset\left(
\mathcal{C}^{0,\alpha}[0,1]\smallsetminus%
{\textstyle\bigcup_{\alpha<\beta}}
\mathcal{C}^{0,\beta}[0,1]\right)  \cup\{0\}\text{.}%
\]

\end{proof}

\begin{corollary}
Given $\alpha\in\lbrack0,1]$ the set $\mathcal{C}^{0,\alpha}%
[0,1]\smallsetminus%
{\textstyle\bigcup_{\alpha<\beta}}
\mathcal{C}^{0,\beta}[0,1]$ is $(\alpha,\mathfrak{c})$-spaceable in
$\mathcal{C}^{0,\alpha}[0,1]$ if and only if $\alpha<\aleph_{0}$.
\end{corollary}

\begin{proof}
For $\alpha=0$ and $\alpha=1$ the result is simple. If $\alpha\in(0,1)$, then%
\[
\bigcup_{\alpha<\lambda}\mathcal{C}^{0,\lambda}[0,1]=\bigcup_{n\in\mathbb{N}%
}\mathcal{C}^{0,\alpha+\frac{1}{n}}[0,1]\text{.}%
\]
Lemma \ref{l6} assures that $(\mathcal{C}^{0,\alpha+\frac{1}{n}%
}[0,1],\Vert\cdot\Vert_{\mathcal{C}^{0,\alpha+\frac{1}{n}}})$ is a Banach
space while Theorem \ref{t10} assures that for each $n\in\mathbb{N}$ there is
$C_{n}>0$ such that $\Vert\cdot\Vert_{\mathcal{C}^{0,\alpha}}\leq C_{n}%
\Vert\cdot\Vert_{\mathcal{C}^{0,\alpha+\frac{1}{n}}}$. Theorem \ref{t11}
assures that $%
{\textstyle\bigcup_{\alpha<\lambda}}
\mathcal{C}^{0,\lambda}[0,1]=%
{\textstyle\bigcup_{n\in\mathbb{N}}}
\mathcal{C}^{0,\alpha+\frac{1}{n}}[0,1]$ has infinite codimension in
$\mathcal{C}^{0,\alpha}[0,1]$. Finally, the result follows from Corollary
\ref{Corollary 2.7}.
\end{proof}


\subsection{Sobolev spaces}


Let us remember the definition of the Sobolev space in one dimension. Given
$m\in\mathbb{N}$ and $1\leq p\leq\infty$ the Sobolev space $W^{m,p}(a,b)$ can
be defined as follows:
\[
W^{m,p}(a,b):=\left\{  u\in L_{p}(a,b):
\begin{array}
[c]{l}%
\text{for each }j\in\left\{  1,\ldots,m\right\}  \text{ there exists }g_{j}\in
L_{p}(a,b)\text{ with }\\%
{\displaystyle\int_{\left(  a,b\right)  }}
u\phi^{\left(  j\right)  }=\left(  -1\right)  ^{j}
{\displaystyle\int_{\left(  a,b\right)  }}
g_{j}\phi\text{ for all }\phi\in C_{c}^{\infty}(a,b)
\end{array}
\right\}
\]
where $C_{c}^{\infty}(a,b)$ denotes the space of compactly supported,
infinitely differentiable functions on $(a,b)$. The function $g_{j}$ is the
well known weak derivative of $j$-order of the function $u$ and it is denoted,
as usual, by $g_{j}\equiv u^{\left(  j\right)  }$.

The norm of $W^{m,p}(a,b)$ is given by
\[
\left\Vert u\right\Vert _{m,p}=\left\Vert u\right\Vert _{p}+\sum
\limits_{j=1}^{m}\left\Vert u^{\left(  j\right)  }\right\Vert _{p}\text{,}
\]
where $\left\Vert \cdot\right\Vert _{p}$ denotes the usual $L_{p}$-norm.

Since $W^{m,r}(a,b)\subset W^{m,s}(a,b)$,
whenever $1\leq s\leq r<\infty$, it follows that the set $%
{\textstyle\bigcup_{q\in\left(  p,\infty\right)  }}
W^{m,q}(a,b)$ is a vector space. Recently Carmona Tapia \textit{et al.} \cite{a1}
proved that, for $1\leq p<\infty$, the set
\[
W^{m,p}(a,b)\smallsetminus\bigcup_{q\in(p,\infty)}W^{m,q}(a,b)
\]
is $\mathfrak{c}$-spaceable.

As application of Corollary \ref{Corollary 2.7}, we have:

\begin{theorem}
For every $1\leq p<\infty$, the set
\[
W^{m,p}(a,b)\smallsetminus\bigcup_{q\in(p,\infty)} W^{m,q}(a,b)
\]
is $\left(  \alpha,\mathfrak{c}\right)  $-spaceable if and only if
$\alpha<\aleph_{0}$.
\end{theorem}

More generally, for an open subset $\Omega$ of $\mathbb{R}^{n}$, $n\geq2$,
recall that
\[
W^{m,p}(\Omega):= \left\{  u\in L_{p}(\Omega):
\begin{array}
[c]{l}%
\displaystyle\text{for each } \alpha\ \text{with } |\alpha|\leq m, \exists
g_{\alpha}\in L_{p}(\Omega) \text{ such that }\\
\displaystyle\int_{\Omega}uD^{\alpha}\phi=\left(  -1\right)  ^{|\alpha|}%
\int_{\Omega}g_{\alpha}\phi\text{ for all }\phi\in C_{c}^{\infty}(\Omega)
\end{array}
\right\}  .
\]
Above we are using the standard multi-index notation: $\alpha=(\alpha
_{1},\ldots,\alpha_{n})$, with $\alpha_{1},\ldots,\alpha_{n}\geq0$ integers,
\[
|\alpha|=\sum_{i=1}^{n} \alpha_{i}\ \ \text{and}\ \ D^{\alpha}\phi
=\frac{\partial^{|\alpha|}\phi}{\partial x^{\alpha_{1}}_{1}\cdots\partial
x^{\alpha_{n}}_{n}}.
\]
We set $D^{\alpha}u = g_{\alpha}$. The space $W^{m,p}(\Omega)$ equipped with
the norm
\[
\left\Vert u\right\Vert _{m,p}=\sum\limits_{0\leq|\alpha|\leq m}\|D^{\alpha
}u\|_{p}
\]
is a Banach space. As noted in \cite{a1}, we can obtain the spaceability of
$W^{m,p}(I)\smallsetminus\bigcup_{q>p}W^{m,q}(I)$ from the spaceability of
$W^{m,p}(I_{1})\smallsetminus\bigcup_{q>p}W^{m,q}(I_{1})$, where
$I=I_{1}\times\cdots\times I_{n}\subset\mathbb{R }^{n}$ is a cartesian product
of bounded open intervals.

With that we have the following result:

\begin{theorem}
If $I=I_{1}\times\cdots\times I_{n}$ is a cartesian product of bounded open
intervals, then $W^{m,p}(I)\smallsetminus\bigcup_{q>p}W^{m,q} (I)$ is
$(\alpha,\mathfrak{c})$-spaceable if and only if $\alpha<\aleph_{0}$.
\end{theorem}

\begin{proof}
Since $I$ is bounded, we have for $p\leq q$ that $L^{q}(I)\subset L^{p}(I)$
with $\|\cdot\|_{p}\leq\lambda(I) ^{\frac{1}{p}-\frac{1}{q}}\|\cdot\|_{q}$.
The inclusions between the spaces $L^{p}(I)$ and the inequality between the
norms immediately imply the same inclusions and inequalities between the norms
of the spaces $W^{m,p}(I)$. Since $\bigcup_{q>p}W^{m,q}(I)=\bigcup
_{r\in\mathbb{N}}W^{m,p+\frac{1}{r }}(I)$, we conclude the result from Corollary \ref{Corollary 2.7}.
\end{proof}

\subsection{Functions of bounded variation}

Let $f\colon\left[  0,1\right]  \longrightarrow\mathbb{R}$ be a continuous
function. The total variation of $f$ is defined by%
\[
\operatorname{Var}\left(  f\right)  =\sup_{P\in\mathcal{P}}\sum_{i=0}%
^{n_{P}-1}\left\vert f\left(  x_{i+1}\right)  -f\left(  x_{i}\right)
\right\vert \text{,}%
\]
where the supremum is taken over the set $\mathcal{P}$ of all partitions
\[
P=\left\{  0=x_{0}<x_{1}<\cdots<x_{n_{P}}=1\right\}
\]
of $\left[  0,1\right]  $. When $\operatorname{Var}\left(  f\right)  <\infty$,
$f$ is said to be of bounded variation. The set $\mathcal{BVC}[0,1]$ of all
continuous functions of bounded variation $f\colon\lbrack0,1]\longrightarrow
\mathbb{R}$ is a vector space and it is easy to check that $\mathcal{BVC}%
[0,1]$, endowed with the norm $\Vert f\Vert=\Vert f\Vert_{\infty
}+\operatorname{Var}(f)$ is a Banach space.

A function $f\colon\left[  0,1\right]  \longrightarrow\mathbb{R}$ is said to
be absolutely continuous if for every $\varepsilon>0$ there is $\delta>0$ such
that, whenever $\left(  \left(  x_{k},y_{k}\right)  \right)  _{k=1}^{n}$ is a
finite sequence of pairwise disjoint open subintervals of $\left[  0,1\right]
$ such that%
\[
\sum_{k=1}^{n}\left(  y_{k}-x_{k}\right)  <\delta\text{,}%
\]
then%
\[
\sum_{k=1}^{n}\left\vert f\left(  y_{k}\right)  -f\left(  x_{k}\right)
\right\vert <\varepsilon\text{.}
\]
The set of all absolutely continuous functions $f\colon\left[  0,1\right]
\longrightarrow\mathbb{R}$ is indicated by $\mathcal{AC}[0,1]$. It is
well-known that $\mathcal{AC}[0,1]$ is a closed linear subspace of
$\mathcal{BVC}[0,1]$ (see \cite{Adams}) and has infinite codimension (see
\cite{Bernal}). Hence, as a consequence of Corollary \ref{Corollary 2.7}, we
have the following result:

\begin{corollary}
The set $\mathcal{BVC}[0,1]\smallsetminus\mathcal{AC}[0,1]$ is $(\alpha
,\mathfrak{c})$-spaceable in $\mathcal{BVC}[0,1]$ if and only if $\alpha<\aleph_{0}$.
\end{corollary}

We know that $\mathcal{BVC}[0,1]$ it is not closed in $\mathcal{C}[0,1]$ and,
if $\mathcal{ND}[0,1]$ is the set of nowhere differentiable continuous
functions, then $\mathcal{ND}[0,1]\subset\mathcal{C}[0,1]\smallsetminus
\mathcal{BVC}[0,1]$, since every $f\in\mathcal{BVC}[0,1]$ is differentiable in
almost everywhere. This shows us that $\mathcal{BVC}[0,1]$ has infinite
codimension in $\mathcal{C}[0,1]$ (because the lineability of $\mathcal{ND}%
[0,1]$ is known) and, obviously, we have the inequality $\Vert\cdot
\Vert_{\infty}\leq\Vert\cdot\Vert$. Hence we can infer the following result:

\begin{corollary}
The set $\mathcal{C}[0,1]\smallsetminus\mathcal{BVC}[0,1]$ is $(\alpha
,\mathfrak{c})$-spaceable in $\mathcal{C}[0,1]$ if and only if $\alpha<\aleph_{0}$.
\end{corollary}

\subsection{Operator spaces}

Let $E_{1},\ldots,E_{m},F$ be Banach spaces. The space of all continuous
$m$-linear operators $T\colon E_{1}\times\cdots\times E_{m}\longrightarrow F$
with the norm%
\[
\left\Vert T\right\Vert =\sup\left\{  \left\Vert T\left(  x_{1},\ldots
,x_{n}\right)  \right\Vert :x_{i}\in B_{E_{i}},\,i=1,\ldots,m\right\}
\]
is denoted by $\mathcal{L}\left(  E_{1},\ldots,E_{m};F\right)  $. When
$E_{1}=\cdots=E_{m}=E$, we simply write $\mathcal{L}\left(  ^{m}E;F\right)  $.
If $m\geq2$, we say that an $m$-linear operator $T\in\mathcal{L}\left(
E_{1},\ldots,E_{m};F\right)  $ is multiple $\left(  r;s\right)  $-summing if
there exists a constant $C>0$ such that%
\begin{equation}
\left(  \sum_{j_{1}=1}^{\infty}\cdots\sum_{j_{m}=1}^{\infty}\left\Vert
T\left(  x_{j_{1}}^{\left(  1\right)  },\ldots x_{j_{m}}^{\left(  m\right)
}\right)  \right\Vert \right)  ^{1/r}\leq C\prod_{k=1}^{m}\left\Vert \left(
x_{j}^{\left(  k\right)  }\right)  _{j=1}^{\infty}\right\Vert _{s,w}
\label{jhlb}%
\end{equation}
for all $\left(  x_{j}^{\left(  k\right)  }\right)  _{j=1}^{\infty}\in\ell
_{s}^{w}\left(  E_{k}\right)  $, $k=1,\ldots,m$. The set of all multiple
$\left(  r;s\right)  $-summing operators from $E_{1}\times\cdots\times E_{m}$
to $F$ is a vector space and is denoted by $\Pi_{\operatorname{mult}\left(
r;s\right)  }\left(  E_{1},\ldots,E_{m};F\right)  $. If $\pi
_{\operatorname{mult}\left(  r;s\right)  }\left(  T\right)  $ represents the
infimum over all constants $C$ in \eqref{jhlb}, then the correspondence
$T\longmapsto\pi_{\operatorname{mult}\left(  r;s\right)  }\left(  T\right)  $
defines a complete norm for $\Pi_{\operatorname{mult}\left(  r;s\right)
}\left(  E_{1},\ldots,E_{m};F\right)  $.

In \cite{AP} the authors proved that, under certain conditions, the set
$\mathcal{L}(^{m}\ell_{p};\mathbb{K})\smallsetminus\Pi_{\operatorname{mult}%
(r;s)}(^{m}\ell_{p};\mathbb{K})$ is spaceable in $\mathcal{L}(^{m}\ell
_{p};\mathbb{K})$. Hence, since $\Pi_{\operatorname{mult}(r;s)}(^{m}\ell
_{p};\mathbb{K})$ is a Banach space with a norm that exceeds the norm of
$\mathcal{L}(^{m}\ell_{p};\mathbb{K})$, we can conclude the following result:

\begin{corollary}
Let $m\in\mathbb{N}$, $m\geq2$, $p\in\left[  2,\infty\right)$ and $1\leq s<p^{\ast}$, where $p^{\ast}$ is the Lebesgue conjugate of $p$. If
\begin{equation}\label{aaabbb}
r<\frac{2ms}{s+2m-ms},
\end{equation}
then $\mathcal{L}(^{m}\ell_{p};\mathbb{K})\smallsetminus\Pi
_{\operatorname{mult}(r;s)}(^{m}\ell_{p};\mathbb{K})$ is $(\alpha,\mathfrak{c})$-spaceable if and only if $\alpha<\aleph_{0}$.
\end{corollary}

The condition \eqref{aaabbb} comes from \cite{AP}, where it was necessary for the authors to prove the spaceability of the set $\mathcal{L}(^{m}\ell_{p};\mathbb{K})\smallsetminus\Pi_{\operatorname{mult}(r;s)}(^{m}\ell_{p};\mathbb{K})$.

Let $1\leq p<\infty$ and let $\Gamma$ be an arbitrary non-void set. We denote
by $\ell_{p}\left(  \Gamma\right)  $ the vector space of all functions
$f\colon\Gamma\longrightarrow\mathbb{K}$ such that $\sum_{\gamma\in\Gamma
}\left\vert f\left(  \gamma\right)  \right\vert ^{p}<\infty$, which becomes a
Banach space with the norm%
\[
\left\Vert f\right\Vert _{p}=\left(  \sum_{\gamma\in\Gamma}\left\vert f\left(
\gamma\right)  \right\vert ^{p}\right)  ^{1/p}\text{,}%
\]
where the sum is defined by%
\[
\sum_{\gamma\in\Gamma}\left\vert f\left(  \gamma\right)  \right\vert ^{p}%
=\sup\left\{  \sum_{\gamma\in L}\left\vert f\left(  \gamma\right)  \right\vert
^{p}:L\text{ is a finite subset of }\Gamma\right\}  \text{.}%
\]
If $E$ and $F$ are Banach spaces, we say that a continuous linear operator
$T\colon E\longrightarrow F$ is absolutely $p$-summing if $\left(  T\left(
x_{j}\right)  \right)  _{j=1}^{\infty}\in\ell_{p}\left(  F\right)  $ whenever
$\left(  x_{j}\right)  _{j=1}^{\infty}\in\ell_{p}^{w}\left(  E\right)  $. The
class of all absolutely $p$-summing linear operators $E\to F$ will be denoted
by $\Pi_{p}(E,F)$.

In \cite{PS} Puglisi and Seoane established, under a certain condition on the
Banach space $E$, that the set $\mathcal{L}(E,\ell_{2})\smallsetminus\Pi
_{1}(E,\ell_{2})$ of non absolutely summing linear operators is lineable.
However, in \cite{FPR} the result obtained by Puglisi and Seoane was improved
in two senses: (i) it is not necessary to assume certain hypotheses on $E$ as
long as $\mathcal{L}(E,\ell_{2})\smallsetminus\Pi_{1}(E,\ell_{2})$ is
non-void, and (ii) the space $\ell_{2}$ can be replaced by $\ell_{2}(\Gamma)$,
an infinite dimensional Hilbert space. The result obtained was as follows:

\begin{theorem}
Let $E$ be a infinite dimensional Banach space and let $\ell_{2}(\Gamma)$ be
an infinite dimensional Hilbert space. If $\mathcal{L}(E,\ell_{2}%
(\Gamma))\smallsetminus\Pi_{1}(E,\ell_{2}(\Gamma))$ is non-empty, then it is
$\operatorname*{card}(\Gamma)$-lineable.
\end{theorem}

With the help of our main result we can improve the above result inserting it
in the framework of $(\alpha,\beta)$-spaceability:

\begin{corollary}
Let $E$ be an infinite dimensional Banach space and let $\ell_{2}(\Gamma)$ an
infinite dimensional Hilbert space. If $\mathcal{L}(E,\ell_{2}(\Gamma
))\smallsetminus\Pi_{1}(E,\ell_{2}(\Gamma))$ is non-empty, then it is
$(\alpha,\mathfrak{c})$-spaceable if and only if $\alpha<\aleph_{0}$.
\end{corollary}

We say that a normed vector space $X$ is finitely representable in a normed
linear space $Y$ if, for each finite-dimensional subspace $X_{n}$ of $X$ and
each number $\lambda>1$, there is an isomorphism $T_{n}$ of $X_{n}$ into $Y$
for which
\[
\lambda^{-1}\left\Vert x\right\Vert \leq\left\Vert T_{n}\left(  x\right)
\right\Vert \leq\lambda\left\Vert x\right\Vert \text{ \ \ \ \ if }x\in
X_{n}\text{.}%
\]
A Banach space $E$ is called super-reflexive if it has the property that no
non-reflexive Banach space is finitely representable in $E$.

In \cite{Kitson} Kitson and Timoney established a remarkable spaceability
result in Fr\'{e}chet spaces with several applications, among them, the
spaceability of the set $\mathcal{K}(E,F)\smallsetminus\bigcup_{1\leq
p<\infty}\Pi_{p}(E,F)$, imposing only that $E$ and $F$ are infinite
dimensional Banach spaces with $E$ super-reflexive. Above, as usual,
$\mathcal{K}(E,F)$ denotes the closed ideal of compact operators in
$\mathcal{L}(E,F)$. The result of Kitson and Timoney allows us to state the
following assertation:

\begin{corollary}
Let $E$ and $F$ be infinite dimensional Banach spaces with $E$
super-reflexive. Then $\mathcal{K}(E,F)\smallsetminus\bigcup_{1\leq p<\infty
}\Pi_{p}(E,F)$ is $(\alpha,\mathfrak{c})$-spaceable in $\mathcal{K}(E,F)$ if and only if $\alpha<\aleph_{0}$.
\end{corollary}

\end{document}